\documentclass[11pt, reqno]{amsart}
\usepackage{enumitem}
\makeatletter
\newcommand{\mylabel}[2]{#2\def\@currentlabel{#2}\label{#1}}
\makeatother
\usepackage{array}
\usepackage[english]{babel}
\usepackage[top=1in, bottom=1in, left=1in, right=1in]{geometry}
\expandafter\let\csname ver@amsthm.sty\endcsname\relax

\usepackage{tikz}

\usepackage{amsmath}
\usepackage{amssymb}
\usepackage{mathdots}
\usepackage{mathtools}

\usepackage{dsfont}
\usepackage{chngpage}

\usepackage{hyperref}
\usepackage{amsthm}
\usepackage[capitalize,noabbrev]{cleveref}

\allowdisplaybreaks

\numberwithin{equation}{section}
\numberwithin{figure}{section}

\newtheorem{thm}{Theorem}[section]
\newtheorem{lemma}[thm]{Lemma}
\newtheorem{cor}[thm]{Corollary}

\newtheorem{Example}[thm]{Example}

\newtheorem{Remark}[thm]{Remark}
\newenvironment{remark}
  {\begin{Remark}\rm}{\hfill$\lozenge$\end{Remark}}
  
\crefname{thm}{Theorem}{Theorems}
\crefname{lemma}{Lemma}{Lemmas}
\crefname{cor}{Corollary}{Corollaries}
\crefname{prop}{Proposition}{Propositions}

\crefname{example}{Example}{Examples}
\crefname{remark}{Remark}{Remarks}

\newcommand{\emailhref}[1]{\email{\href{#1}{#1}}}
\newcommand{\M}{\operatorname{M}}
\newcommand{\wt}{\operatorname{wt}}

\title[Lozenge Tilings of Hexagons with Intrusions II: Shuffling Phenomenon]{Lozenge Tilings of Hexagons with Intrusions II:\\ Shuffling Phenomenon}

\author[Seok Hyun Byun]{Seok Hyun Byun}\emailhref{sbyun@amherst.edu}
\address{Department of Mathematics, Amherst College, Amherst, MA 01002, U.S.A.}

\author[Tri Lai]{Tri Lai}\emailhref{tlai3@unl.edu}
\address{Department of Mathematics, University of Nebraska – Lincoln, Lincoln, NE 68588, U.S.A.}
\thanks{T.L. was supported in part by Simons Collaboration Grant (\#585923).}

\begin{document}

\begin{abstract}
The enumeration of lozenge tilings of hexagons with holes has been studied intensively in recent years. Researchers tried to find shapes and positions of holes in hexagonal regions so that the number of lozenge tilings of the resulting regions is given by a simple product formula. In the present work, we consider new regions that are hybrids of regions studied by the first author (\textit{hexagons with intrusions}) and Ciucu (\textit{F-cored hexagons}). Then, we show that the tiling generating functions of these new regions under a certain weight are given by simple product formulas. To give a proof, we present \textit{shuffling theorems for lozenge tilings of hexagons with intrusions}, which give simple relations between the tiling generating functions of two related hexagonal regions with intrusions.
\end{abstract}

\maketitle

\section{Introduction}\label{sec:1}

A plane partition is one of the central objects in algebraic combinatorics, due to its connections with other important objects in combinatorics and other research fields like probability and statistical mechanics. There are many enumeration results on plane partitions that pose certain properties. Among them, one of the most significant results is MacMahon's Theorem on boxed plane partitions. There are many equivalent ways to state MacMahon's theorem, and here we recall a lozenge tiling version\footnote{A \textit{lozenge} is a union of two adjacent unit triangles that share an edge. A \textit{lozenge tiling} of a region is a collection of lozenges that covers the region without gaps and overlaps.} of it (this is due to the bijection of David and Tomei \cite{david1989problem}). The number of lozenge tilings of a hexagon with sides of length $a,b,c,a,b,$ and $c$ in cyclic order is given by
\begin{equation}\label{eaa}
    \prod_{i=1}^{a}\prod_{j=1}^{b}\prod_{k=1}^{c}\frac{i+j+k-1}{i+j+k-2}=\frac{H(a)H(b)H(c)H(a+b+c)}{H(a+b)H(b+c)H(c+a)},
\end{equation}
where\footnote{Throughout this paper, empty product is understood as $1$. Thus, $H(0)=1$.} $H(n)\coloneqq\prod_{i=0}^{n-1}i!$ for nonnegative integers $n$.

Motivated by this beautiful product formula, researchers have tried to generalize MacMahon's theorem to hexagonal regions with holes whose number of lozenge tilings is also given by a simple product formula. Some generalizations in this direction can be found in \cite{byun2022lozenge,byunlai2025lozenge,Ciucu2005PP1,ciucu2017other,CEKZ2001Core,CK2013dual,CL2019doubly,fulmek2023damaged,Lai2017boundaryfern,Lai2017threedents,Lai2020centralholedent,Lai2022offcentral,lai2019shamrockaxis,rosengren2016selberg}. Among them, we recall two generalizations: one is the result of the first author, and the other is the result of Ciucu. Consider a triangular lattice oriented so that one family of lattice lines is vertical. In \cite{byun2022lozenge}, the first author considered a centrally symmetric hexagon of sides of length $a,b,c,a,b,$ and $c$ clockwise from the left. Then, on the perpendicular bisector of the left side of length $a$, the first author removed an even number of left-aligned unit triangles (the collection of these deleted unit triangles was called \textit{an intrusion}) from the hexagon and showed that the number of lozenge tilings of the resulting regions is given by a simple product formula (see the left picture in Figure \ref{faa}). On the other hand, in \cite{ciucu2017other}, Ciucu considered a structure called \textit{a fern}, which is a string of triangles of arbitrary sizes that alternate orientations, touch at corners, and are lined up along an axis. Then he removed a fern from a hexagonal region (he called the region \textit{F-cored hexagon}) and showed that if we place a fern hole at a specific position, then the number of lozenge tilings of the resulting region is given by a simple product formula (see the right picture in Figure \ref{faa}). The motivation of this paper is the following question.\\

\begin{figure}
    \centering
    \includegraphics[width=0.7\textwidth]{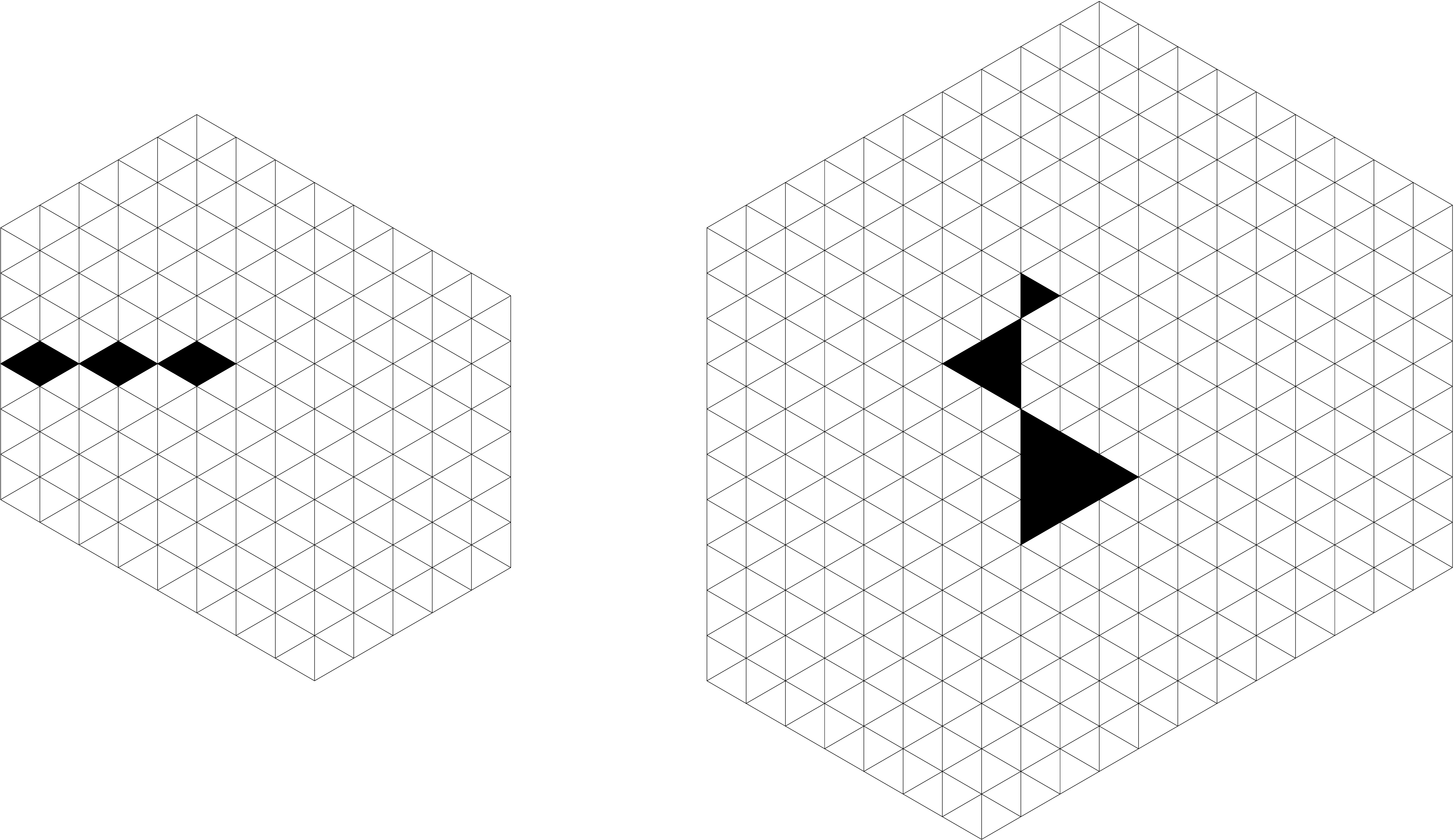}
    \caption{Examples of a hexagon with an intrusion (left) and a hexagon with a fern removed from its center (right). The numbers of lozenge tilings of these regions are given by simple product formulas.}
    \label{faa}
\end{figure}

\textit{Can we make both an intrusion and a fern hole from a hexagonal region so that the number of lozenge tilings of the resulting region is given by a simple product formula?}\\

There are several ways to remove these two structures from the hexagonal region. One possible way is to make an intrusion in the hexagon first and then place the fern at the end of the intrusion, as shown in the left picture in Figure \ref{fab}. Unfortunately, numerical data suggested that the number of lozenge tilings of the resulting regions is not given by a simple product formula in general, as the prime factorizations of their tiling numbers contain large prime factors in many cases. However, if we place the intrusion on the perpendicular bisector of one side (as the first author did in \cite{byun2022lozenge}) and if we remove a \textit{symmetric fern}\footnote{A fern is \textit{symmetric} if it is symmetric about the perpendicular bisector of the line segment whose end points are the two extreme points of the fern on the axis} at the end of the intrusion (see the right picture in Figure \ref{fab} for an example), then we observed that the prime factorizations of the number of lozenge tilings of the resulting regions always consist of small prime numbers, which is a good indication of the existence of product formulas. We also observed that if we assign a certain weight to the lozenges in the same region, the tiling generating function of the region under that weight also seemed to be fully factorized.

\begin{figure}
    \centering
    \includegraphics[width=0.7\textwidth]{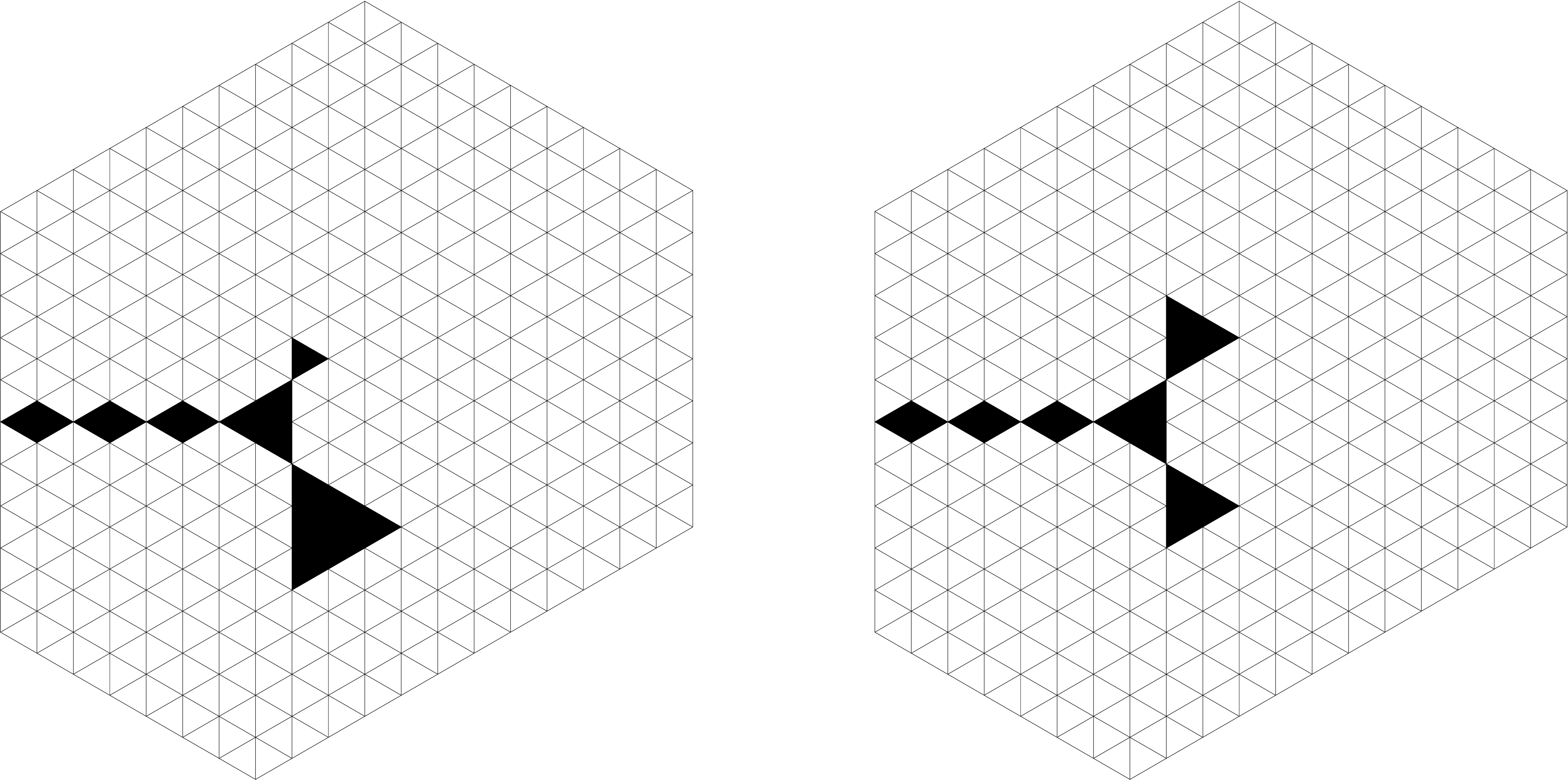}
    \caption{Examples of a hexagon with an intrusion and an asymmetric fern removed (left) and a hexagon with an intrusion and a symmetric fern removed (right).}
    \label{fab}
\end{figure}

Once we made this observation, we tried to find a way to conjecture formulas for the number of lozenge tilings (or tiling generating functions) of these new regions. The problem we encounter is that these regions have so many parameters, so it is almost impossible to guess the correct formulas. Since we could not guess the formulas, we were not able to prove this observation using an inductive argument based on Kuo's graphical condensation method \cite{kuo2004applications}, which has been widely used in recent developments in the field of tiling enumerations. To overcome this, we used the idea of the second author and Rohatgi from \cite{lai2019shuffling}. In that paper, the authors came up with the \textit{shuffling theorem for lozenge tilings of doubly-dented hexagons}, which gave a simple proof of Ciucu's result on the \textit{F-cored hexagon} mentioned earlier. Using a similar idea, we could come up with \textit{shuffling theorems for lozenge tilings of hexagons with intrusions} (Theorems \ref{tca} and \ref{tcb}). For the proof of the theorems, we use the idea that was used independently by the first author \cite{byun2022shuffling} and Fulmek \cite{fulmek2021shuffling} to give a simple proof of the original shuffling theorem in \cite{lai2019shuffling}. Once we prove these new shuffling theorems, we could prove our observations, and this is presented in Section \ref{sec:3}.

This paper is organized as follows. In Section \ref{sec:2}, we define five families of regions and their tiling generating functions. Each of these five regions is obtained from a hexagonal region by deleting an intrusion and a symmetric fern. In Section \ref{sec:3}, we present shuffling theorems for lozenge tilings of hexagons with intrusions (Theorems \ref{tca} and \ref{tcb}). Then, we state the exact enumeration results for the tiling generating functions of the five regions introduced in the previous section using the shuffling theorems (Theorem \ref{tcd}). In Section \ref{sec:4}, we recall some results that we utilize in the proof of Theorems \ref{tca} and \ref{tcb}. In Section \ref{sec:5}, we give a proof of Theorems \ref{tca} and \ref{tcb}. For the completeness of the paper, we also add a proof of Lemma \ref{tda} in Appendix \ref{sec:App} at the end of the paper.

\section{Five hexagonal regions with intrusions and symmetric ferns removed and their tiling generating functions}\label{sec:2}

In this section, we introduce five families of regions. As mentioned in the previous section, each region is obtained from a hexagonal region by deleting an intrusion and a symmetric fern. The enumeration formulas depend on the parity of the size of the central triangle of the fern and its orientation, and this is why we consider several families of regions. Before we introduce the regions, we first define intrusions and symmetric ferns in detail.

\begin{figure}
    \centering
    \includegraphics[width=0.7\textwidth]{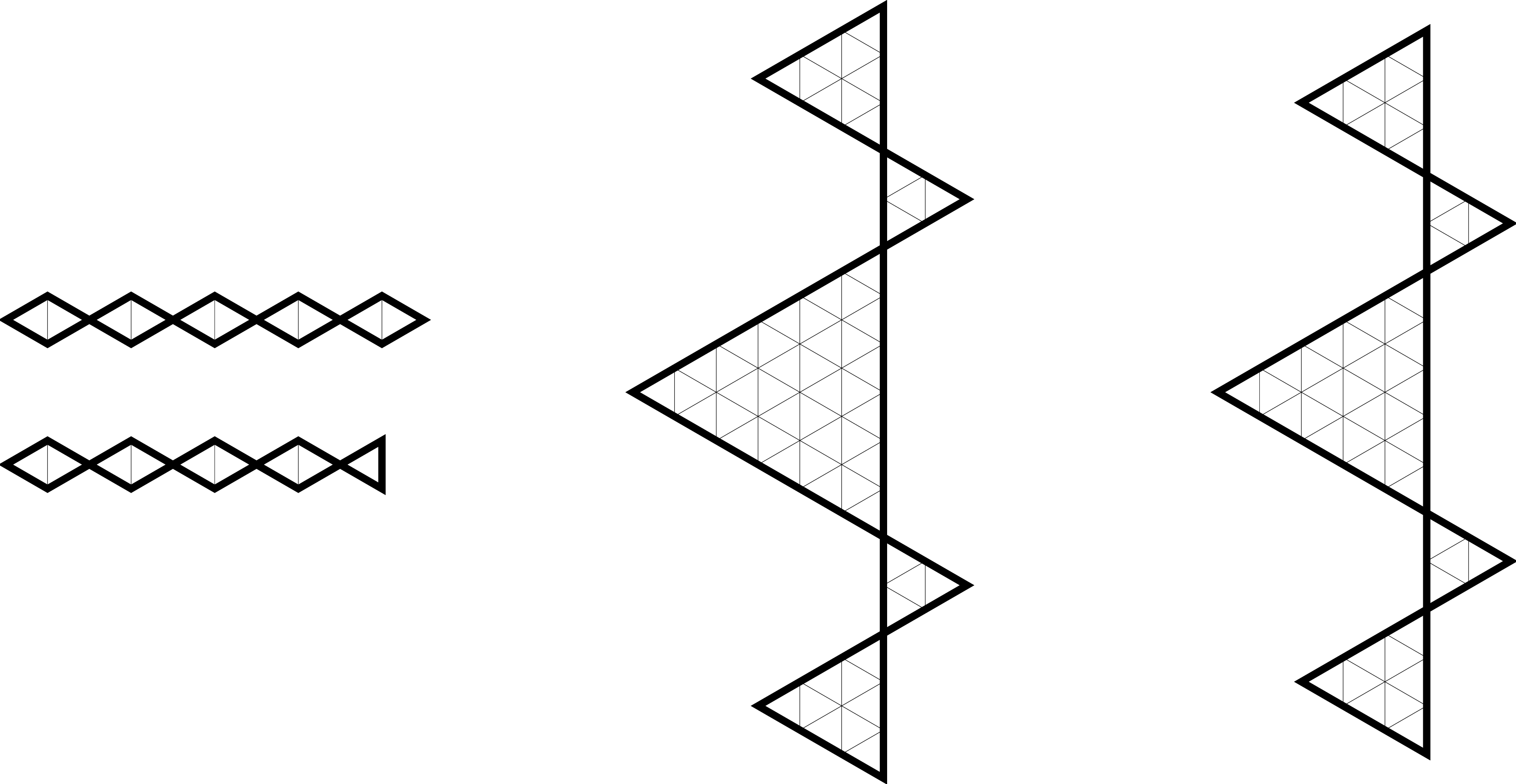}
    \caption{Intrusions of length $10$ and $9$ (top left and bottom left) and symmetric ferns $F_{sym}(6,2,3)$ (center) and $F_{sym}(5,2,3)$ (right).}
    \label{fba}
\end{figure}

For a nonnegative integer $n$, an \textit{intrusion of length $n$} is a collection of $n$ unit triangles as presented in the left pictures in Figure \ref{fba}. As one can see from the pictures, the orientation of the rightmost unit triangle depends on the parity of $n$. For a nonnegative integer $k$, positive integers $a_{1},\ldots,a_{2k}$ and a nonnegative integer $a_{2k+1}$, we define a \textit{symmetric fern $F_{sym}(a_{1},a_{2},\ldots,a_{2k+1})$} as a string of triangles of size $a_{2k+1},\ldots,a_{2},a_{1},a_{2},\ldots$, and $a_{2k+1}$ one after another that alternate orientations, touch at corners, and are lined up along an axis (we call this axis a \textit{fern axis}). See the two pictures on the right in Figure \ref{fba}. The central triangle of size $a_{1}$ of the symmetric fern $F_{sym}(a_{1},a_{2},\ldots,a_{2k+1})$ is a \textit{core} of the symmetric fern, and we assume that it is left-pointing. The condition $a_{2k+1}\in\mathbb{Z}_{\geq0}$ allows the symmetric fern to have freedom on the orientation of the furthest triangle from the core: it has the same orientation as the core if $a_{2k+1}>0$ and has a different orientation as the core if $a_{2k+1}=0$. Note that when $k=0$, the symmetric fern $F_{sym}(a_{1})$ is a triangle of size $a_1$. The mirror images of $F_{sym}(a_{1},a_{2},\ldots,a_{2k+1})$ across a vertical axis is a \textit{flipped symmetric fern} and denoted by $\overline{F}_{sym}(a_{1},a_{2},\ldots,a_{2k+1})$.

Consider $x,y,z,w,k\in\mathbb{Z}_{\geq0}$, $a_1\ldots,a_{2k}\in\mathbb{Z}_{>0}$, and $a_{2k+1}\in\mathbb{Z}_{\geq0}$. We further assume that $z\geq w$ without loss of generality. We also define $a_{o}\coloneqq\sum_{i=0}^{k}a_{2i+1}$, $a_{e}\coloneqq\sum_{i=1}^{k}a_{2i}$, and $a\coloneqq a_{o}+a_{e}=\sum_{i=1}^{2k+1}a_{i}$. We now construct five families of regions $A(x,y,z,w;a_1,\ldots,a_{2k+1})$,$\ldots$, $E(x,y,z,w;a_1,\ldots,a_{2k+1})$. As will be seen in the definitions,
\begin{itemize}
    \item $y$ determines the length of the intrusion,
    \item $a_1,\ldots,a_{2k+1}$ determine the shape of the (flipped) symmetric fern, and 
    \item $x,z,w$ determine the lengths of the sides of the boundary hexagon.
\end{itemize}

\begin{figure}
    \centering
    \includegraphics[width=0.84\textwidth]{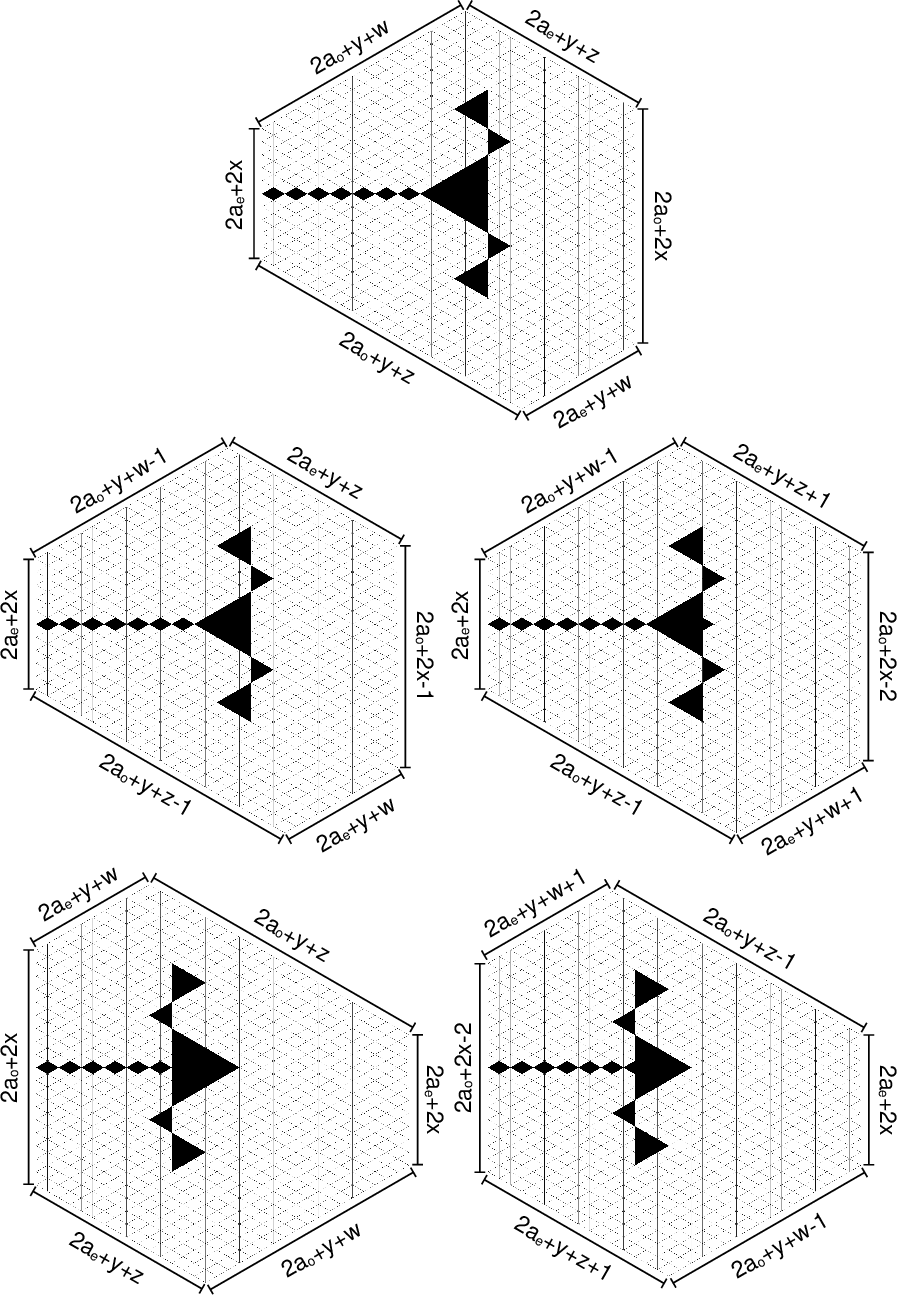}
    \caption{Five regions $A(x,y,z,w;a_1,a_2,a_3)$ (top), $B(x,y,z,w;a_1,a_2,a_3)$ (middle left), $C(x,y,z,w;a_1,a_2,a_3)$ (middle right), $D(x,y,z,w;a_1,a_2,a_3)$ (bottom left), and $E(x,y,z,w;a_1,a_2,a_3)$ (bottom right), where $x=3$, $y=4$, $z=7$, $w=2$, $a_1=3$, $a_2=2$, and $a_3=3$ (thus $a_o=a_1+a_3=6$ and $a_e=a_2=2$).}
    \label{fbb}
\end{figure}

We first define $A(x,y,z,w;a_1,\ldots,a_{2k+1})$. Consider a hexagon with sides of length $2a_{e}+2x$, $2a_{o}+y+w$, $2a_{e}+y+z$, $2a_{o}+2x$, $2a_{e}+y+w$, and $2a_{o}+y+z$ clockwise from the left side. From this hexagon, consider the perpendicular bisector of the left side and delete a left-aligned intrusion of length $2(a_{o}-a_1)+2y$ on it. Then, we further delete a symmetric fern $F_{sym}(2a_1,a_2,\ldots,a_{2k+1})$ from the resulting region so that the vertex of the core not on the fern axis touches the rightmost vertex of the intrusion. For example, see the picture at the top in Figure \ref{fbb}. Notice that the first parameter in $F_{sym}(2a_1,a_2,\ldots,a_{2k+1})$ is $2a_1$, not $a_1$. This means that the central triangle in the removed fern has an even length. There will be analogous parity conditions on the lengths of the cores of the removed ferns in the other four regions.

To define $B(x,y,z,w;a_1,\ldots,a_{2k+1})$, we consider a hexagon with sides of length $2a_{e}+2x$, $2a_{o}+y+w-1$, $2a_{e}+y+z$, $2a_{o}+2x-1$, $2a_{e}+y+w$, and $2a_{o}+y+z-1$ clockwise from the left side. From this hexagon, consider the perpendicular bisector of the left side and delete a left-aligned intrusion of length $2(a_{o}-a_1)+2y$ on it. Then, we further delete a symmetric fern $F_{sym}(2a_1-1,a_2,\ldots,a_{2k+1})$ so that the intrusion and the symmetric fern touch as in the construction of $A(x,y,z,w;a_1,\ldots,a_{2k+1})$. For example, see the left picture in the middle of Figure \ref{fbb}.

The construction of $C(x,y,z,w;a_1,\ldots,a_{2k+1})$ is very similar to that of $B(x,y,z,w;a_1,\ldots,a_{2k+1})$: we first consider a hexagon with sides of length $2a_{e}+2x$, $2a_{o}+y+w-1$, $2a_{e}+y+z+1$, $2a_{o}+2x-2$, $2a_{e}+y+w+1$, and $2a_{o}+y+z-1$ clockwise from the left side. We then delete the same intrusion of length $2(a_{o}-a_1)+2y$ and the same symmetric fern $F_{sym}(2a_1-1,a_2,\ldots,a_{2k+1})$ from the perpendicular bisector of the left side the same way as we construct $B(x,y,z,w;a_1,\ldots,a_{2k+1})$. As an additional step, we delete a right-pointing unit triangle that is on the perpendicular bisector of the left side and shares an edge with the core of the fern. For example, see the right picture in the middle of Figure \ref{fbb}.

To construct $D(x,y,z,w;a_1,\ldots,a_{2k+1})$, we consider a hexagon with sides of length $2a_{o}+2x$, $2a_{e}+y+w$, $2a_{o}+y+z$, $2a_{e}+2x$, $2a_{o}+y+w$, and $2a_{e}+y+z$ clockwise from the left side. We then delete an intrusion of length $2a_{e}+2y$ and a flipped symmetric fern $\overline{F}_{sym}(2a_{1},a_{2},\ldots,a_{2k+1})$ as shown in the left picture at the bottom of Figure \ref{fbb}. In particular, the rightmost vertex of the intrusion touches the midpoint of the side of the core of the fern that is on the fern axis.

Lastly, we define $E(x,y,z,w;a_1,\ldots,a_{2k+1})$. We first consider a hexagon with sides of length $2a_{o}+2x-2$, $2a_{e}+y+w+1$, $2a_{o}+y+z-1$, $2a_{e}+2x$, $2a_{o}+y+w-1$, and $2a_{e}+y+z+1$ clockwise from the left. We then delete an intrusion of length $2a_{e}+2y+1$ and a flipped symmetric fern $\overline{F}_{sym}(2a_{1}-1,a_{2},\ldots,a_{2k+1})$ as shown in the right picture at the bottom in Figure \ref{fbb}. In particular, the midpoint of the rightmost vertical side of the intrusion matches the midpoint of the side of the core of the fern that is on the fern axis.

\begin{figure}
    \centering
    \includegraphics[width=0.80\textwidth]{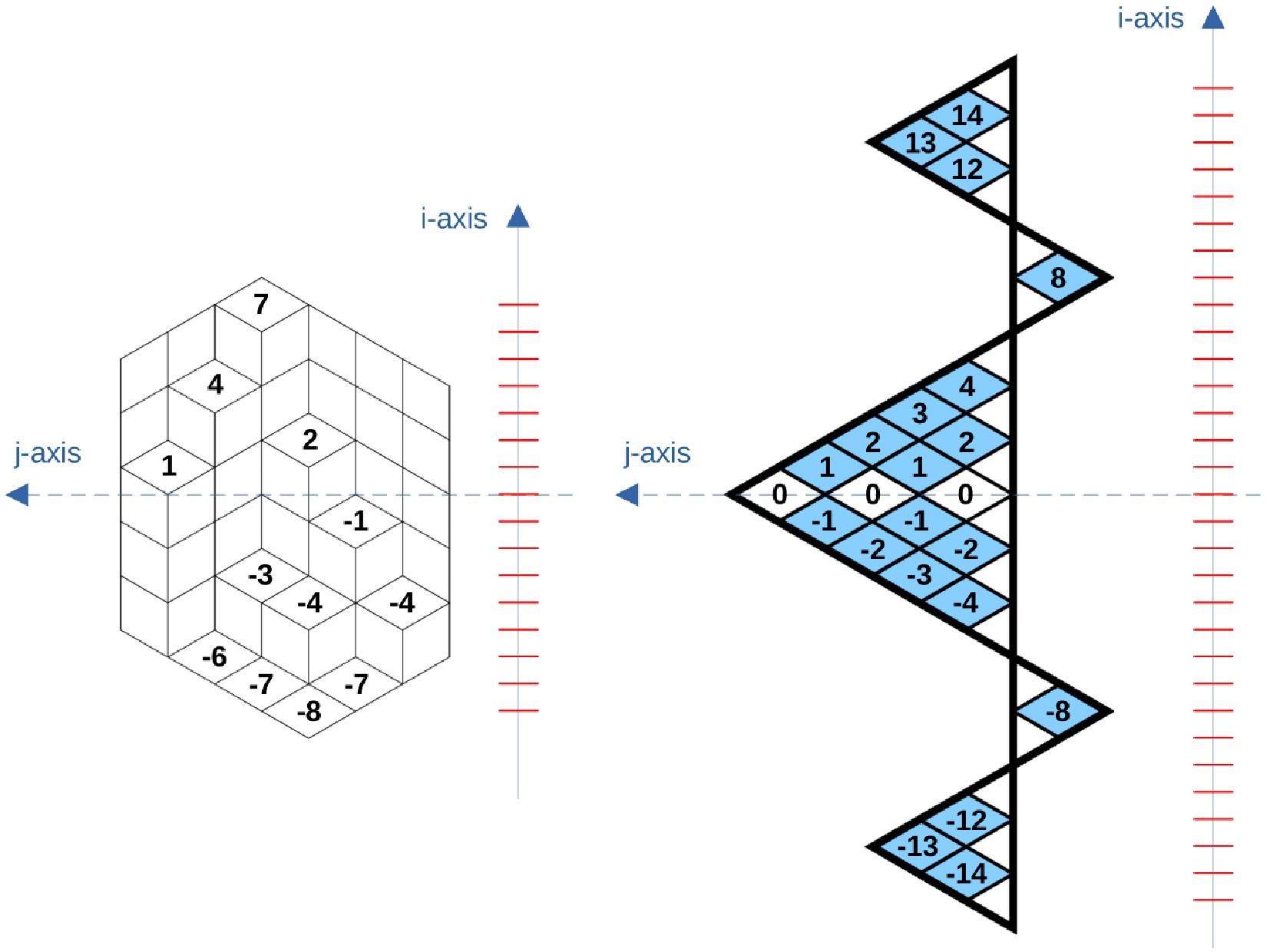}
    \caption{A lozenge tiling of a hexagonal region on the $(i,j)$-coordinate system (left) and the symmetric fern $F_{sym}(6,2,3)$ on the $(i,j)$-coordinate system (right). In both pictures, the $i$-coordinates of the centers of horizontal lozenges are marked. If a horizontal lozenge is marked by $n$, then it is weighted by $\frac{q^{n}+q^{-n}}{2}$ and every lozenge with no mark is weighted by $1$.}
    \label{fbc}
\end{figure}

Every hexagonal region in this paper will be considered as if it were placed on the $(i,j)$-coordinate system, such that 
\begin{itemize}
    \item the perpendicular bisector of the left side of the hexagon is on the $j$ axis and
    \item the half of the side length of unit triangles is the unit length of the $i$-axis.
\end{itemize}

We then assign weights to lozenges in hexagonal regions, according to their orientation and the $i$-coordinate of the center of the lozenges. According to the slope of their long diagonal, there are three types of lozenges. We call them \emph{positive}, \emph{negative}, or \emph{horizontal lozenges} if their long diagonals have positive, negative, or horizontal slope, respectively. The way we assign weight to lozenges is as follows. We assign weight $1$ to every positive and negative lozenge. For a horizontal lozenge, if its center has $i$-coordinate $n$, then we assign a weight $\frac{q^{n}+q^{-n}}{2}$ to the lozenge. See the left picture in Figure \ref{fbc} that illustrates this weight assignment. Under this weight assignment, the \textit{weight of a lozenge tiling} of a region is the product of the weights of all lozenges that constitute the tiling. The \textit{tiling generating function of a region $R$} is the sum of weights of all lozenge tilings of $R$, and we denote it by $\M_{q}(R)$.

We also define a \textit{$q$-weight} of a symmetric fern $F_{sym}(a_{1},a_{2},\ldots,a_{2k+1})$, which we denote it by $\wt(F_{sym}(a_{1},a_{2},\ldots,a_{2k+1}))$. To do that, we place the symmetric fern on the $(i,j)$-coordinate system so that the symmetric fern is symmetric across the $j$-axis. Then, the weight of the symmetric fern is the product of the weights of all horizontal lozenges contained in the symmetric fern, where lozenges are weighted as explained in the previous paragraph (see the right picture in Figure \ref{fbc}). Since every horizontal lozenge on the $j$-axis has weight $\frac{q^{0}+q^{0}}{2}=1$, the weight of a symmetric fern can also be thought of as the product of all horizontal lozenges in the symmetric fern that are not on the $j$-axis (see the shaded horizontal lozenges in the right picture in Figure \ref{fbc}. This observation will be used later). The weight of a flipped symmetric fern $\overline{F}_{sym}(a_{1},a_{2},\ldots,a_{2k+1})$, denoted by $\wt(\overline{F}_{sym}(a_{1},a_{2},\ldots,a_{2k+1}))$, is defined in an exactly the same way. In particular, $\wt(F_{sym}(a_{1},a_{2},\ldots,a_{2k+1}))=\wt(\overline{F}_{sym}(a_{1},a_{2},\ldots,a_{2k+1}))$.

We further define two sets $P_o$ and $P_e$ associated to the symmetric fern $F_{sym}(a_{1},a_{2},\ldots,a_{2k+1})$ and the flipped symmetric fern $\overline{F}_{sym}(a_{1},a_{2},\ldots,a_{2k+1})$. $P_{o}$ encodes the $i$-coordinates of the centers of the unit triangles along the fern axis that are contained in the triangles of size $a_{1}, a_{3},\ldots, a_{2k+1}$. Similarly, $P_{e}$ encodes the $i$-coordinates of the centers of the unit triangles along the fern axis that are contained in the triangles of size $a_{2}, a_{4},\ldots, a_{2k}$. 
For any set of real numbers $R$, we say $R$ is \textit{symmetric} if $R=-R\coloneqq\{-r|r\in R\}$. Note that these two sets $P_{o}$ and $P_{e}$ are symmetric (i.e. $P_{o}=-P_{o}$ and $P_{e}=-P_{e}$) because $j$-axis cut through the symmetry axis of the symmetric fern. Furthermore, the two sets consist of odd integers (or even integers) if $a_1$ is even (or odd). For example, the sets corresponding to the symmetric fern on the right picture in Figure \ref{fbc} are $P_{o}=\{-15,-13,-11,-5,-3,-1,1,3,5,11,13,15\}$ and $P_{e}=\{-9,-7,7,9\}$. We then define\footnote{Strictly speaking, the sets $P_{o}$, $P_{e}$, $Q_{o}$, and $Q_{e}$ depend on the symmetric fern $F_{sym}(a_{1},a_{2},\ldots,a_{2k+1})$ (or the flipped symmetric fern $\overline{F}_{sym}(a_{1},a_{2},\ldots,a_{2k+1})$), and thus depend on the parameters $a_{1},a_{2},\ldots,a_{2k+1}$. Therefore, it would make more sense to denote them by $P_{o}(a_{1},a_{2},\ldots,a_{2k+1})$, $P_{e}(a_{1},a_{2},\ldots,a_{2k+1})$, $Q_{o}(a_{1},a_{2},\ldots,a_{2k+1})$, and $Q_{e}(a_{1},a_{2},\ldots,a_{2k+1})$ instead of $P_{o}$, $P_{e}$, $Q_{o}$, and $Q_{e}$, respectively. However, these notations are too long, so we use brief notations $P_{o}$, $P_{e}$, $Q_{o}$, and $Q_{e}$, and keep in mind that they depend on the (flipped) symmetric fern.} $Q_{o}\coloneqq\frac{1}{2}P_{o}=\{\frac{1}{2}p~|~p\in P_{o}\}$ and $Q_{e}\coloneqq\frac{1}{2}P_{e}=\{\frac{1}{2}p~|~p\in P_{e}\}$.

The main enumeration result of this paper (Theorem \ref{tcd}) gives explicit product formulas for the ratios between one of $\M_{q}(A(x,y,z,w;a_1,\ldots,a_{2k+1})),\ldots, \M_{q}(E(x,y,z,w;a_1,\ldots,a_{2k+1}))$ and one of $\M_{q}(A(x,y,z,w;a)), \M_{q}(B(x,y,z,w;a)), \M_{q}(C(x,y,z,w;a))$ up to a multiplicative factor independent from $x,y,z,w$ (see Equations \eqref{ece}-\eqref{eci}). We state these results in Section \ref{sec:3} because their formulas require results and notations from Theorems \ref{tca} and \ref{tcb}.

As a consequence of Theorem \ref{tcd}, we can deduce the tiling generating functions of the regions $A(x,y,z,w;a_1,\ldots,a_{2k+1}),\ldots, E(x,y,z,w;a_1,\ldots,a_{2k+1})$ using the previous results of the authors in \cite{byunlai2025lozenge}. This is because the regions $A(x,y,z,w;a)$, $B(x,y,z,w;a)$, and $C(x,y,z,w;a)$ have the same tiling generating functions as the regions\footnote{In all three regions, the sequence ``$0,\ldots,0$" consists of $y-1$ copies of $0$s.} $H_{e}(x,w,z,y;0,\ldots,0,a)$, $H_{o}(x,w,z,y;0,\ldots,0,a-1)$, and $H_{e}(x,w,z,y+1;0,\ldots,0,a-1,0)$ appearing in \cite{byunlai2025lozenge}, respectively (see \cite{byunlai2025lozenge} for the detailed definition of these regions). This can be easily checked as follows. For example, observe the leftmost strip of $A(x,y,z,w;a)$ (see the picture on the top in Figure \ref{fbd}). One can see that the leftmost strip of the region is uniquely tiled by lozenges that have weight $1$ (this is because they are not horizontal. See the shaded lozenges). Thus, deleting the leftmost strip from $A(x,y,z,w;a)$ does not change the tiling generating function, and the resulting region is precisely the region $H_{e}(x,w,z,y;0,\ldots,0,a)$ whose tiling generating function is obtained in \cite{byunlai2025lozenge}. Similarly, by deleting the leftmost strips from $B(x,y,z,w;a)$ and $C(x,y,z,w;a)$, one gets $H_{o}(x,w,z,y;0,\ldots,0,a-1)$ and $H_{e}(x,w,z,y+1;0,\ldots,0,a-1,0)$ whose tiling generating functions are presented in \cite{byunlai2025lozenge}, respectively. This explains why $\M_{q}(A(x,y,z,w;a))$, $\M_{q}(B(x,y,z,w;a)$, and $\M_{q}(C(x,y,z,w;a))$ are already known by the results in \cite{byunlai2025lozenge}.

\begin{figure}
    \centering
        \includegraphics[width=0.9\textwidth]{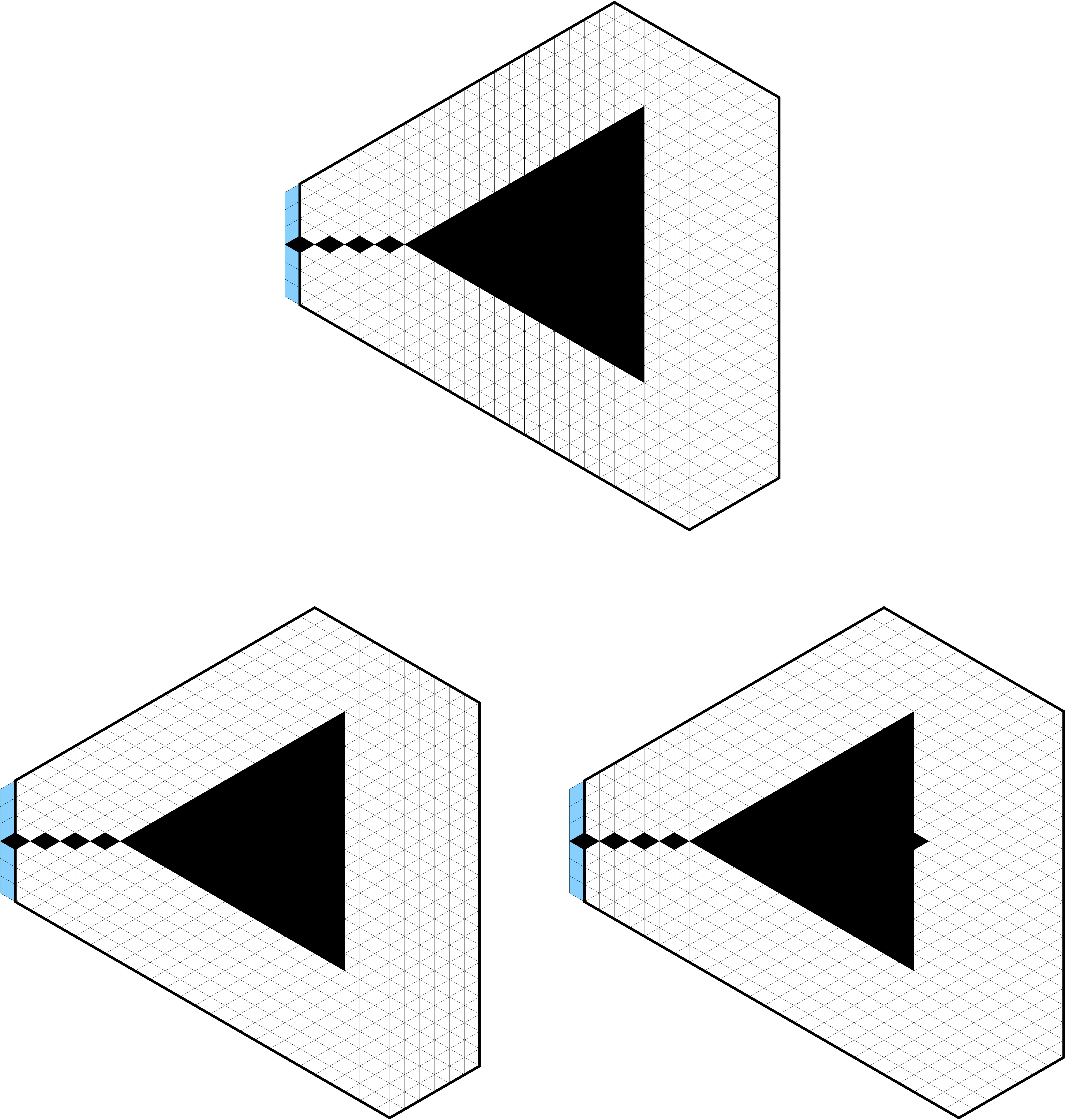}
    \caption{$A(3,4,7,2;8)$ (top), $B(3,4,7,2;8)$ (bottom left), $C(3,4,7,2;8)$ (bottom right). After removing the leftmost strip, they become $H_{e}(3,2,7,4;0,0,0,8)$, $H_{o}(3,2,7,4;0,0,0,7)$, and $H_{e}(3,2,7,5;0,0,0,7,0)$ of \cite{byunlai2025lozenge}, respectively.}
    \label{fbd}
\end{figure}

More precise statements of the theorems will be provided in Section \ref{sec:3}, after we present shuffling theorems for hexagons with intrusions in Theorems \ref{tca} and \ref{tcb}.

\section{Shuffling theorems for hexagons with intrusions}\label{sec:3}

In this section, we first state shuffling theorems for hexagons with intrusions (Theorems \ref{tca} and \ref{tcb}). The proof of these theorems is postponed to Section \ref{sec:5} as it requires some preliminary results, which will be presented in Section \ref{sec:4}. We then finish the section by stating the tiling generating functions results for the five regions introduced in the previous section (Theorems \ref{tcd}) and giving a detailed proof of one of them, as the proofs for the five regions are very similar.

\begin{figure}
    \centering
    \includegraphics[width=0.9\textwidth]{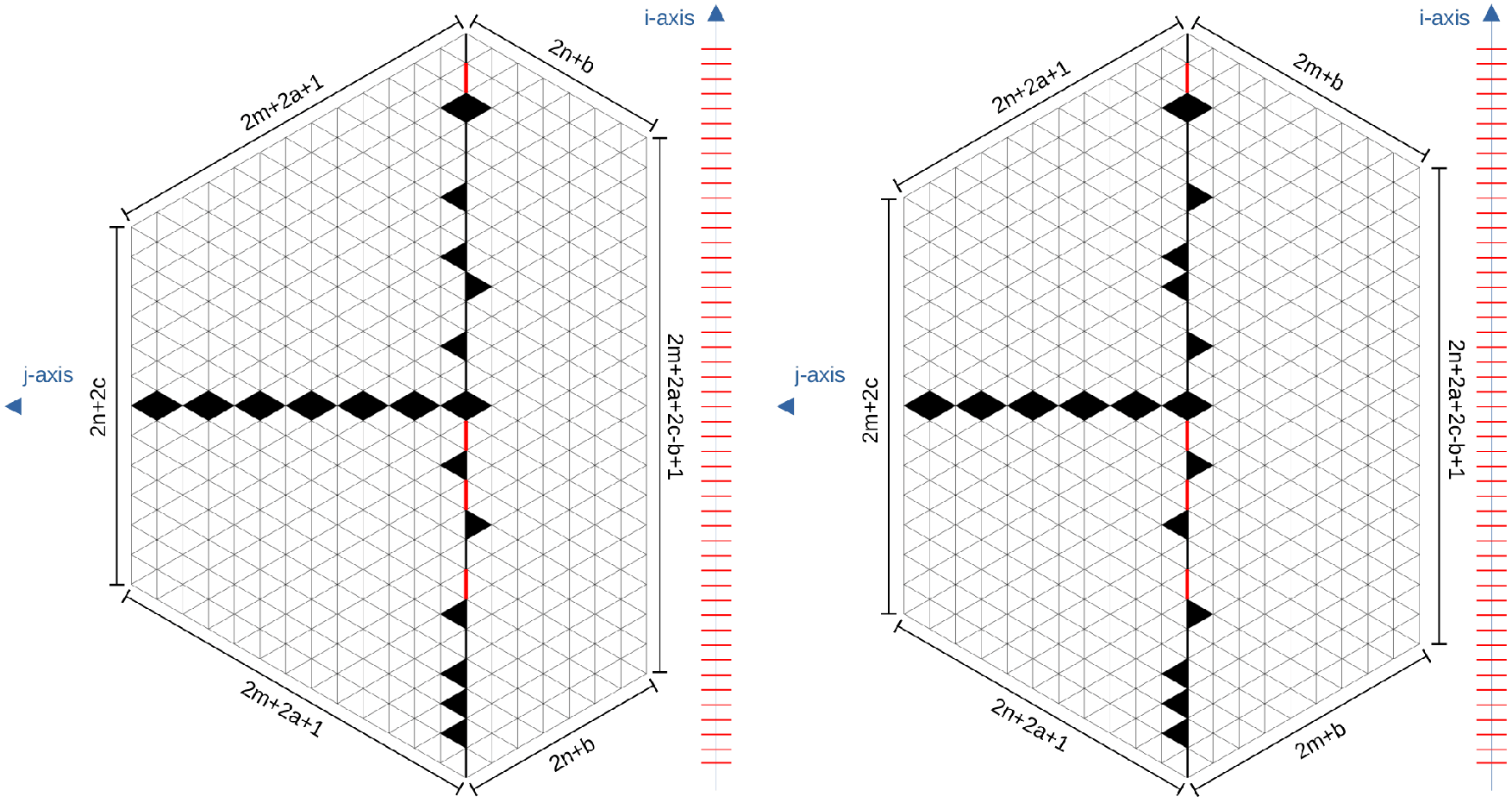}
    \caption{$H_{2,1,4,5,5}(\{-11, -10, -9, -7, -2, 2, 5, 7, 10\},\{-4, 0, 4, 10\},\{-6, -3, -1, 11\})$ (left) and $H_{1,2,4,5,5}(\{-11, -10, -9, -4, 4, 5, 10\},\{-7, -2, 0, 2, 7, 10\},\{-6, -3, -1, 11\})$ (right) on the $(i,j)$-planes. Positions of barriers are indicated by red color. The picture on the right is obtained from the left picture by flipping left-pointing unit triangles labeled by $-7,-2,2,7$ and right-pointing unit triangles labeled by $-4,4$.}
    \label{fca}
\end{figure}

For any nonnegative integers $a,b,c,m,$ and $n$ such that $2a+2c+1\geq b$, we consider a symmetric hexagon with sides of lengths $2n+2c$, $2m+2a+1$, $2n+b$, $2m+2a+2c-b+1$, $2n+b$, and $2m+2a+1$ clockwise from the left. Then we remove $2m+2a+1$ leftmost unit triangles on its horizontal symmetry axis (we are making an intrusion of length $2m+2a+1$). Note that this intrusion contacts the vertical diagonal of the hexagon. Label the unit segments on the diagonal by $-(m+n+a+c), -(m+n+a+c)+1,\ldots, m+n+a+c-1, m+n+a+c$ from the bottom to the top. Now we remove some left-pointing unit triangles and right-pointing unit triangles that share a side with the vertical diagonal. We give these unit triangles the same labels as their vertical sides. Let $L_1$ (and $R_1$) be the label set of the removed left-pointing (and right-pointing) unit triangles. Since the left-pointing unit triangle labeled by $0$ is part of the intrusion, note that $L_1$ cannot contain $0$. We will flip $2m$ left-pointing unit triangles and $2n$ right-pointing unit triangles through the vertical diagonal, subject to the restriction that if two deleted unit triangles form a unit lozenge-shaped hole, we do not flip either of them. Hence, we assume that $|L_1\setminus R_1|\geq 2m$ and $|R_1\setminus L_1|\geq 2n$. A unit segment on the triangular grid is called a \textit{barrier} if covering it by a lozenge is prohibited. We impose barriers on some unit segments along the vertical diagonal, and let $B$ denote the set of barrier labels. We denote the resulting region by $H_{m,n,a,b,c}(L_1, R_1, B)$ (see the left picture in Figure \ref{fca}). Since $H_{m,n,a,b,c}(L_1, R_1, B)$ and $H_{m,n,a,b,c}(L_1, R_1, B\setminus(L_1\cup R_1))$ have the same set of lozenge tilings, without loss of generality, we assume $B\cap (L_1\cup R_1)=\varnothing$.

Now, we flip $2m$ removed left-pointing unit triangles and $2n$ removed right-pointing unit triangles along the vertical diagonal in a \textit{symmetric way}. (This means that we flip a removed triangle with label $k$ and the same type of removed triangle with label $-k$ at the same time. Note that the index sets $L_1$ and $R_1$ do not have to be symmetric, although we are flipping unit triangles in a symmetric way.) An important part is that, when we flip these unit triangles, the boundary hexagon and the length of the intrusion will be changed accordingly.  Let $L_2$ (and $R_2$) be the set of the labels of the removed left-pointing (and right-pointing) unit triangles after flipping. Once flipping is done, we consider the region $H_{n,m,a,b,c}(L_2, R_2, B)$ (please compare the two pictures in Figure \ref{fca}). Note that the parameters $m$ and $n$ are interchanged, so the boundary hexagon and the length of the intrusion are modified accordingly.

We then consider the tiling generating functions of these two regions. To do that, we put these regions on the $(i,j)$-plane as described in the previous section. More precisely, we put the regions so that the $j$-axis cuts through the horizontal symmetry axis of the boundary hexagons (see two pictures in Figure \ref{fca}). Every lozenge is weighted according to its orientation and $i$-coordinate of the center, as explained in Section \ref{sec:2}. Let $\M_{q}(H_{m,n,a,b,c}(L_1, R_1, B))$ and $\M_{q}(H_{n,m,a,b,c}(L_2, R_2, B))$ be the tiling generating functions of the two regions under this weight assignment.

To state our result, we need to set some notations. Throughout this paper, we use the following $q$-analogue of positive integers $n$ and its factorial $n!$: $\langle n\rangle_q\coloneqq \frac{q^n-q^{-n}}{q-q^{-1}}$ and $\langle n\rangle_q!\coloneqq \langle 1\rangle_q\langle 2\rangle_q\cdots\langle n\rangle_q$. We also define $\langle n\rangle_q^+\coloneqq \frac{q^n+q^{-n}}{2}$ for integers $n$. Lastly, for any finite set $S \subset \mathbb{Z}$ (or $S \subset \mathbb{Z}+\frac{1}{2}$), we define $\displaystyle\Delta_{1,1,q}(S)\coloneqq \prod_{s_1, s_2 \in S, s_1 < s_2}\langle s_2+s_1\rangle_{q}^{+}\langle s_2-s_1\rangle_{q}$.

\begin{thm}\label{tca}
If $\M_{q}(H_{m,n,a,b,c}(L_1, R_1, B))\neq 0$, then
\begin{equation}\label{eca}
    \frac{\M_{q}(H_{n,m,a,b,c}(L_2, R_2, B))}{\M_{q}(H_{m,n,a,b,c}(L_1, R_1, B))}=\frac{\Bigg[\displaystyle\prod_{i=1}^{m+a}\langle 2i-1\rangle_q!\Bigg]^2}{\Bigg[\displaystyle\prod_{i=1}^{n+a}\langle 2i-1\rangle_q!\Bigg]^{2}}\cdot\frac{\displaystyle\prod_{i=1}^{2n+b-1}{\langle i\rangle_{q}!}}{\displaystyle\prod_{i=1}^{2m+b-1}{\langle i\rangle_{q}!}}\cdot\sqrt{\frac{\displaystyle\prod_{z\in L_{2}}\frac{\langle 2|z|\rangle_{q}}{4}}{\displaystyle\prod_{z\in L_{1}}\frac{\langle 2|z|\rangle_{q}}{4}}}\cdot\frac{\Delta_{1,1,q}(L_{2})\Delta_{1,1,q}(R_{2})}{\Delta_{1,1,q}(L_{1})\Delta_{1,1,q}(R_{1})}.
\end{equation}
\end{thm}
The right side of the above equation involves a square root of a ratio of Laurent polynomials in $q$. It turns out that after canceling out common factors, the expression in the square root sign is the square of a certain ratio of Laurent polynomials in $q$. When we take the square root, we choose the sign so that the leading coefficients of both numerator and denominator are positive. The same comment applies to Theorem \ref{tcb}.

There also exists a similar theorem for the case when the vertical diagonal has an even length. To state the second theorem, we recall the following set notations: for any set of real numbers $A$ and a real number $r$, $A+r\coloneqq \{a+r|a\in A\}$ and $-A\coloneqq \{-a|a\in A\}$.

\begin{figure}
    \centering
    \includegraphics[width=0.9\textwidth]{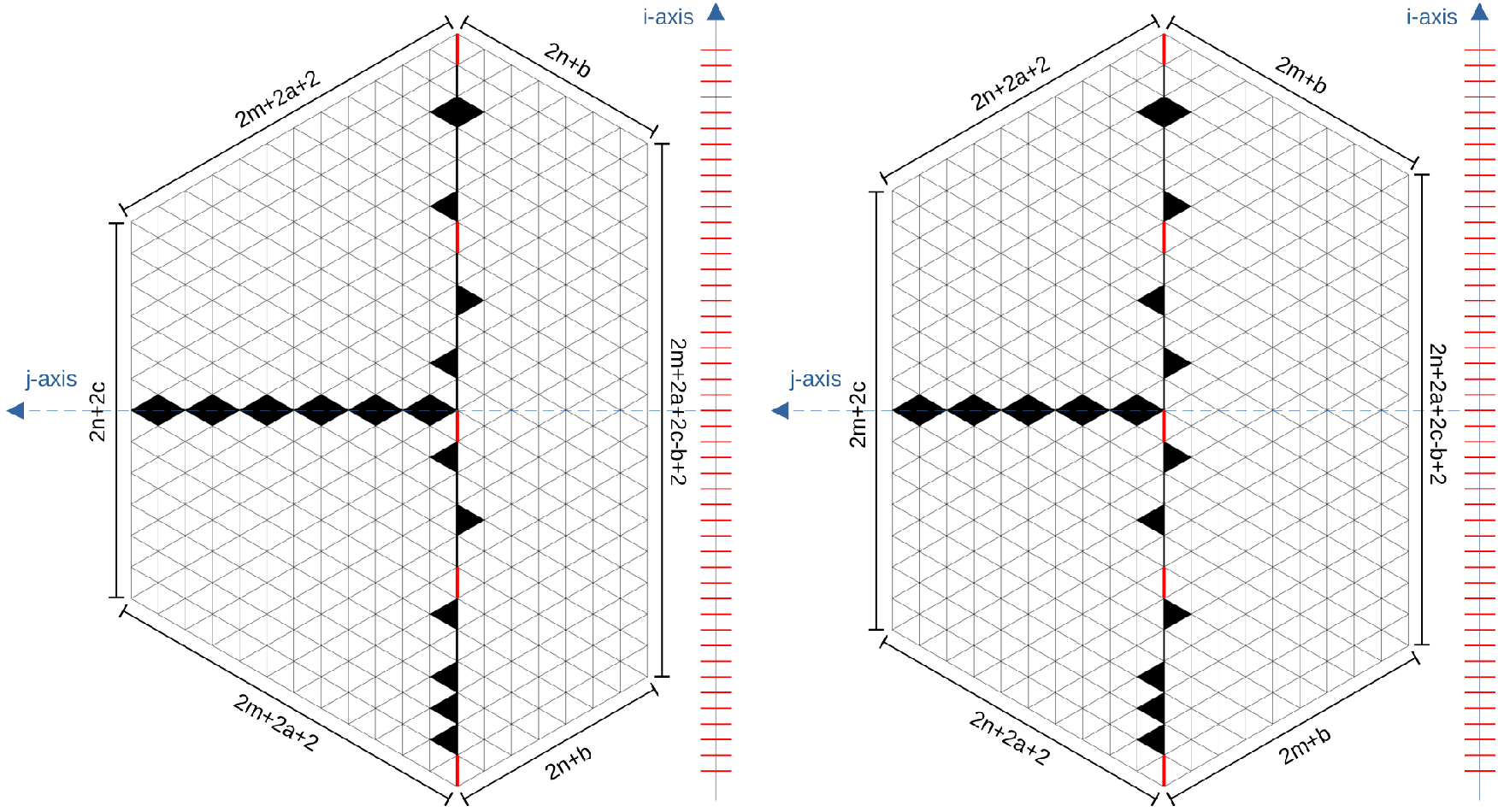}
    \caption{$H'_{2,1,3,5,5}(\{-\frac{21}{2}, -\frac{19}{2}, -\frac{17}{2}, -\frac{13}{2}, -\frac{3}{2}, \frac{3}{2}, \frac{13}{2}, \frac{19}{2}\},\{-\frac{7}{2}, \frac{7}{2}, \frac{19}{2}\},\{-\frac{23}{2}, -\frac{11}{2}, -\frac{1}{2}, \frac{11}{2}, \frac{23}{2}\})$ (left) and $H'_{1,2,3,5,5}(\{-\frac{21}{2}, -\frac{19}{2}, -\frac{17}{2}, -\frac{7}{2}, \frac{7}{2}, \frac{19}{2}\},\{-\frac{13}{2}, -\frac{3}{2}, \frac{3}{2}, \frac{13}{2}, \frac{19}{2}\},\{-\frac{23}{2}, -\frac{11}{2}, -\frac{1}{2}, \frac{11}{2}, \frac{23}{2}\})$ (right) on the $(i,j)$-planes. Positions of barriers are indicated by red color. The picture on the right is obtained from the left picture by flipping left-pointing unit triangles labeled by $-\frac{13}{2},-\frac{3}{2},\frac{3}{2},\frac{13}{2}$ and right-pointing unit triangles labeled by $-\frac{7}{2},\frac{7}{2}$.}
    \label{fcb}
\end{figure}

For any nonnegative integers $a,b,c,m$ and $n$ such that $2a+2c+2\geq b$, we consider a symmetric hexagon of side length $2n+2c$, $2m+2a+2$, $2n+b$, $2m+2a+2c-b+2$, $2n+b$ and $2m+2a+2$ clockwise from the left. Then we remove $2m+2a+2$ leftmost unit triangles (an intrusion of length $2m+2a+2$) on the horizontal symmetry axis. This intrusion touches the vertical diagonal of the hexagon. Label the unit segments on the diagonal by $-(m+n+a+c)-\frac{1}{2}, -(m+n+a+c)+\frac{1}{2},\dots, m+n+a+c-\frac{1}{2}, m+n+a+c+\frac{1}{2}$ from the bottom to the top. Now we remove some left-pointing unit triangles and right-pointing unit triangles that share their sides with the vertical diagonal. Let $L_1$ (and $R_1$) be the set of the labels of unit segments on the vertical diagonal that are shared with the removed left-pointing (and right-pointing) unit triangles. Again, we assume that $|L_1\setminus R_1|\geq 2m$ and $|R_1\setminus L_1|\geq 2n$. As before, we also put barriers on some of the unit segments on the vertical diagonal and let $B$ be the set of the labels of the barriers. We call the region that we just described $H'_{m,n,a,b,c}(L_1, R_1, B)$ (see the left picture in Figure \ref{fcb}).
Now, we flip $2m$ removed left-pointing unit triangles and $2n$ removed right-pointing unit triangles across the vertical diagonal in a symmetric way. Let $L_2$ (and $R_2$) be the set of the labels of the removed left-pointing (and right-pointing) unit triangles after flipping. Again, once flipping is done, we consider the region $H'_{n,m,a,b,c}(L_2, R_2, B)$ (compare the two pictures in Figure \ref{fcb}). We assign weight to lozenges in these regions in the same way as before and denote their tiling generating functions by $\M_{q}(H'_{m,n,a,b,c}(L_1, R_1, B))$ and $\M_{q}(H'_{n,m,a,b,c}(L_2, R_2, B))$, respectively.

\begin{thm}\label{tcb}
If $\M_{q}(H'_{m,n,a,b,c}(L_1, R_1, B))\neq 0$, then
\begin{equation}\label{ecb}
    \frac{\M_{q}(H'_{n,m,a,b,c}(L_2, R_2, B))}{\M_{q}(H'_{m,n,a,b,c}(L_1, R_1, B))}=\frac{\Bigg[\displaystyle\prod_{i=1}^{m+a+1}\langle 2i-2\rangle_q!\Bigg]^2}{\Bigg[\displaystyle\prod_{i=1}^{n+a+1}\langle 2i-2\rangle_q!\Bigg]^2}\cdot\frac{\displaystyle\prod_{i=1}^{2n+b-1}{\langle i\rangle_{q}!}}{\displaystyle\prod_{i=1}^{2m+b-1}{\langle i\rangle_{q}!}}\cdot\sqrt{\frac{\displaystyle\prod_{s\in L_{1}}\langle 2|s|\rangle_q}{\displaystyle\prod_{s\in L_{2}}\langle 2|s|\rangle_q}}\cdot\frac{\Delta_{1,1,q}(L_{2})\Delta_{1,1,q}(R_{2})}{\Delta_{1,1,q}(L_{1})\Delta_{1,1,q}(R_{1})}.
\end{equation}
\end{thm}

Since the proof of Theorems \ref{tca} and \ref{tcb} requires some preliminary results, we postpone the proof of the theorems to Section \ref{sec:5}.

\begin{remark}\label{tcc}
    Observe that in \eqref{eca}, the right side of the equation is independent of the set $B$, which encodes the position of the barriers on the vertical diagonal. In particular, if we denote $H_{m,n,a,b,c}(L_1, R_1, \varnothing)$ by $H_{m,n,a,b,c}(L_1, R_1)$, then from our observation, we have
    \begin{equation}\label{ecc}
        \frac{\M_{q}(H_{n,m,a,b,c}(L_2, R_2, B))}{\M_{q}(H_{m,n,a,b,c}(L_1, R_1, B))}=\frac{\M_{q}(H_{n,m,a,b,c}(L_2, R_2))}{\M_{q}(H_{m,n,a,b,c}(L_1, R_1))}.
    \end{equation}
    Similarly, the right side of \eqref{ecb} is independent of the set $B$. Thus, if we denote $H'_{m,n,a,b,c}(L_1, R_1, \varnothing)$ by $H'_{m,n,a,b,c}(L_1, R_1)$, then we have
    \begin{equation}\label{ecd}
        \frac{\M_{q}(H'_{n,m,a,b,c}(L_2, R_2, B))}{\M_{q}(H'_{m,n,a,b,c}(L_1, R_1, B))}=\frac{\M_{q}(H'_{n,m,a,b,c}(L_2, R_2))}{\M_{q}(H'_{m,n,a,b,c}(L_1, R_1))}.
    \end{equation}
    Our proof of Theorems \ref{tca} and \ref{tcb} will give a clear explanation of why the formulas in \eqref{eca} and \eqref{ecb} do not depend on the set $B$. These two brief notations $H_{m,n,a,b,c}(L_1, R_1)$ and $H'_{m,n,a,b,c}(L_1, R_1)$ will be used when we state Theorem \ref{tcd} below.
\end{remark}

As mentioned earlier, we now state exact enumeration results for the tiling generating functions of the five regions $A(x,y,z,w;a_1,\ldots,a_{2k+1}),\ldots, E(x,y,z,w;a_1,\ldots,a_{2k+1})$ introduced in the previous section. The result is given in terms of formulas in Theorems \ref{tca} and \ref{tcb}. We use the standard notation $[n]\coloneqq\{1,\ldots,n\}$ for a positive integer $n$ and $[0]\coloneqq\varnothing$. Recall that $a=\sum_{i=1}^{2k+1}a_{i}$ and $Q_{o}\coloneqq\frac{1}{2}P_{o}$ and $Q_{e}\coloneqq\frac{1}{2}P_{e}$, where $P_{o}$ and $P_{e}$ were defined in Section 2 (they encode the $i$-coordinates of the center of the unit triangles in the fern).

\begin{thm}\label{tcd}
    Let $x,y,z,w,k\in\mathbb{Z}_{\geq0}$, $a_1\ldots,a_{2k}\in\mathbb{Z}_{>0}$, and $a_{2k+1}\in\mathbb{Z}_{\geq0}$ such that $z\geq w$.    
    \begin{enumerate}
    \item 
    The tiling generating function of the region $A(x,y,z,w;a_1,\ldots,a_{2k+1})$ satisfies
    \begin{equation}\label{ece}
    \begin{aligned}
        &\frac{\M_{q}(A(x,y,z,w;a_1,\ldots,a_{2k+1}))}{\M_{q}(A(x,y,z,w;a))}\cdot \frac{\wt(F_{sym}(2a_{1},a_{2},\ldots,a_{2k+1}))}{\wt(F_{sym}(2a))}\\
        =&
        \begin{cases}
        \displaystyle \frac{\M_q(H'_{0,a_{e},a_{o}+y-1,z+w,x-y+z}(Q_{o},Q_{e}\cup X\cup Y))}{\M_q(H'_{a_{e},0,a_{o}+y-1,z+w,x-y+z}(Q_{o}\cup Q_{e},X\cup Y))} & \text{if $y < w$}\\[10pt]
        \displaystyle \frac{\M_q(H'_{0,a_{e},a_{o}+y-1,z+w,x-y+z}(Q_{o}\cup Y,Q_{e}\cup X))}{\M_q(H'_{a_{e},0,a_{o}+y-1,z+w,x-y+z}(Q_{o}\cup Q_{e}\cup Y,X))} & \text{if $w\leq y\leq z$}\\[10pt]
        \displaystyle \frac{\M_q(H'_{0,a_{e},a_{o}+y-1,z+w,x}(Q_{o}\cup X\cup Y,Q_{e}))}{\M_q(H'_{a_{e},0,a_{o}+y-1,z+w,x}(Q_{o}\cup Q_{e}\cup X\cup Y,\varnothing))} & \text{if $z < y$}
        \end{cases},
    \end{aligned}
    \end{equation}
    where $X=\begin{cases}
        [z-y]-\frac{2a+2x+2z+1}{2}, & \text{if $y\leq z$}\\
        [y-z]-\frac{2a+2x+2y+1}{2}, & \text{if $z<y$}
    \end{cases}$ and
    
    $Y=\begin{cases}
        [2a+2x+z+w]\setminus[2a+2x+y+z]-\frac{2a+2x+2z+1}{2}, & \text{if $y<w$}\\
        [2a+2x+y+z]\setminus[2a+2x+z+w]-\frac{2a+2x+2z+1}{2}, & \text{if $w\leq y\leq z$}\\
        [2a+2x+2y]\setminus[2a+2x+y+w]-\frac{2a+2x+2y+1}{2}, & \text{if $z<y$}
    \end{cases}$.

    \vspace{5mm}
    \item 
    The tiling generating function of the region $B(x,y,z,w;a_1,\ldots,a_{2k+1})$ satisfies
    \begin{equation}\label{ecf}
    \begin{aligned}
        &\frac{\M_{q}(B(x,y,z,w;a_1,\ldots,a_{2k+1}))}{\M_{q}(B(x,y,z,w;a))}\cdot\frac{\wt(F_{sym}(2a_{1}-1,a_{2},\ldots,a_{2k+1}))}{\wt(F_{sym}(2a-1))}\\
        =&
        \begin{cases}
        \displaystyle \frac{\M_q(H_{0,a_{e},a_{o}+y-1,z+w,x-y+z}(Q_{o},Q_{e}\cup X\cup Y))}{\M_q(H_{a_{e},0,a_{o}+y-1,z+w,x-y+z}(Q_{o}\cup Q_{e},X\cup Y))} & \text{if $y < w$}\\[10pt]
        \displaystyle \frac{\M_q(H_{0,a_{e},a_{o}+y-1,z+w,x-y+z}(Q_{o}\cup Y,Q_{e}\cup X))}{\M_q(H_{a_{e},0,a_{o}+y-1,z+w,x-y+z}(Q_{o}\cup Q_{e}\cup Y,X))} & \text{if $w\leq y\leq z$}\\[10pt]
        \displaystyle \frac{\M_q(H_{0,a_{e},a_{o}+y-1,z+w,x}(Q_{o}\cup X\cup Y,Q_{e}))}{\M_q(H_{a_{e},0,a_{o}+y-1,z+w,x}(Q_{o}\cup Q_{e}\cup X\cup Y,\varnothing))} & \text{if $z < y$}
        \end{cases},
    \end{aligned}
    \end{equation}
    where $X=\begin{cases}
        [z-y]-(a+x+z), & \text{if $y\leq z$}\\
        [y-z]-(a+x+y), & \text{if $z<y$}
    \end{cases}$ and
    
    $Y=\begin{cases}
        [2a+2x+z+w-1]\setminus[2a+2x+y+z-1]-(a+x+z), & \text{if $y<w$}\\
        [2a+2x+y+z-1]\setminus[2a+2x+z+w-1]-(a+x+z), & \text{if $w\leq y\leq z$}\\
        [2a+2x+2y-1]\setminus[2a+2x+y+w-1]-(a+x+y), & \text{if $z<y$}
    \end{cases}$.

    \vspace{5mm}
    \item 
    The tiling generating function of the region $C(x,y,z,w;a_1,\ldots,a_{2k+1})$ satisfies
    \begin{equation}\label{ecg}
    \begin{aligned}
        &\frac{\M_{q}(C(x,y,z,w;a_1,\ldots,a_{2k+1}))}{\M_{q}(C(x,y,z,w;a))}\cdot\frac{\wt(F_{sym}(2a_{1}-1,a_{2},\ldots,a_{2k+1}))}{\wt(F_{sym}(2a-1))}\\
        =&
        \begin{cases}
        \displaystyle \frac{\M_q(H_{0,a_{e},a_{o}+y-1,z+w+1,x-y+z}(Q_{o},Q_{e}\cup X\cup Y\cup\{0\}))}{\M_q(H_{a_{e},0,a_{o}+y-1,z+w+1,x-y+z}(Q_{o}\cup Q_{e},X\cup Y\cup\{0\}))} & \text{if $y < w$}\\[10pt]
        \displaystyle \frac{\M_q(H_{0,a_{e},a_{o}+y-1,z+w+1,x-y+z}(Q_{o}\cup Y,Q_{e}\cup X\cup\{0\}))}{\M_q(H_{a_{e},0,a_{o}+y-1,z+w+1,x-y+z}(Q_{o}\cup Q_{e}\cup Y, X\cup\{0\}))} & \text{if $w\leq y\leq z$}\\[10pt]
        \displaystyle \frac{\M_q(H_{0,a_{e},a_{o}+y-1,z+w+1,x}(Q_{o}\cup X\cup Y,Q_{e}\cup\{0\}))}{\M_q(H_{a_{e},0,a_{o}+y-1,z+w+1,x}(Q_{o}\cup Q_{e}\cup X\cup Y,\{0\}))} & \text{if $z < y$}
        \end{cases},
    \end{aligned}
    \end{equation}
    where $X=\begin{cases}
        [z-y]-(a+x+z), & \text{if $y\leq z$}\\
        [y-z]-(a+x+y), & \text{if $z<y$}
    \end{cases}$ and
    
    $Y=\begin{cases}
        [2a+2x+z+w-1]\setminus[2a+2x+y+z-1]-(a+x+z), & \text{if $y<w$}\\
        [2a+2x+y+z-1]\setminus[2a+2x+z+w-1]-(a+x+z), & \text{if $w\leq y\leq z$}\\
        [2a+2x+2y-1]\setminus[2a+2x+y+w-1]-(a+x+y), & \text{if $z<y$}
    \end{cases}$.

    \vspace{5mm}
    \item 
    The tiling generating function of the region $D(x,y,z,w;a_1,\ldots,a_{2k+1})$ satisfies
    \begin{equation}\label{ech}
    \begin{aligned}    
        &\frac{\M_{q}(D(x,y,z,w;a_1,\ldots,a_{2k+1}))}{\M_{q}(A(x,y,z,w;a))}\cdot \frac{\wt(F_{sym}(2a_{1},a_{2},\ldots,a_{2k+1}))}{\wt(F_{sym}(2a))}\\
        =&
        \begin{cases}
        \displaystyle \frac{\M_q(H'_{0,a_{o},a_{e}+y-1,z+w,x-y+z}(Q_{e},Q_{o}\cup X\cup Y))}{\M_q(H'_{a_{o},0,a_{e}+y-1,z+w,x-y+z}(Q_{o}\cup Q_{e},X\cup Y))} & \text{if $y < w$}\\[10pt]
        \displaystyle \frac{\M_q(H'_{0,a_{o},a_{e}+y-1,z+w,x-y+z}(Q_{e}\cup Y,Q_{o}\cup X))}{\M_q(H'_{a_{o},0,a_{e}+y-1,z+w,x-y+z}(Q_{o}\cup Q_{e}\cup Y, X))} & \text{if $w\leq y\leq z$}\\[10pt]
        \displaystyle \frac{\M_q(H'_{0,a_{o},a_{e}+y-1,z+w,x}(Q_{e}\cup X\cup Y,Q_{o}))}{\M_q(H'_{a_{o},0,a_{e}+y-1,z+w,x}(Q_{o}\cup Q_{e}\cup X\cup Y,\varnothing))} & \text{if $z < y$}
        \end{cases},
    \end{aligned}
    \end{equation}
    where $X=\begin{cases}
        [z-y]-\frac{2a+2x+2z+1}{2}, & \text{if $y\leq z$}\\
        [y-z]-\frac{2a+2x+2y+1}{2}, & \text{if $z<y$}
    \end{cases}$ and
    
    $Y=\begin{cases}
        [2a+2x+z+w]\setminus[2a+2x+y+z]-\frac{2a+2x+2z+1}{2}, & \text{if $y<w$}\\
        [2a+2x+y+z]\setminus[2a+2x+z+2]-\frac{2a+2x+2z+1}{2}, & \text{if $w\leq y\leq z$}\\
        [2a+2x+2y]\setminus[2a+2x+y+w]-\frac{2a+2x+2y+1}{2}, & \text{if $z<y$}
    \end{cases}$.

    \vspace{5mm}
    \item 
    The tiling generating function of the region $E(x,y,z,w;a_1,\ldots,a_{2k+1})$ satisfies
    \begin{equation}\label{eci}
    \begin{aligned}
        &\frac{\M_{q}(E(x,y,z,w;a_1,\ldots,a_{2k+1}))}{\M_{q}(C(x,y,z,w;a))}\cdot \frac{\wt(F_{sym}(2a_{1}-1,a_{2},\ldots,a_{2k+1}))}{\wt(F_{sym}(2a-1))}\\
        =&
        \begin{cases}
        \displaystyle \frac{\M_q(H_{0,a_{o}-1,a_{e}+y,z+w+1,x-y+z}(Q_{e},Q_{o}\cup X\cup Y))}{\M_q(H_{a_{o}-1,0,a_{e}+y,z+w+1,x-y+z}((Q_{o}\setminus\{0\})\cup Q_{e},X\cup Y\cup\{0\}))} & \text{if $y < w$}\\[10pt]
        \displaystyle \frac{\M_q(H_{0,a_{o}-1,a_{e}+y,z+w+1,x-y+z}(Q_{e}\cup Y,Q_{o}\cup X))}{\M_q(H_{a_{o}-1,0,a_{e}+y,z+w+1,x-y+z}((Q_{o}\setminus\{0\})\cup Q_{e}\cup Y, X\cup\{0\}))} & \text{if $w\leq y\leq z$}\\[10pt]
        \displaystyle \frac{\M_q(H_{0,a_{o}-1,a_{e}+y,z+w+1,x}(Q_{e}\cup X\cup Y,Q_{o}))}{\M_q(H_{a_{o}-1,0,a_{e}+y,z+w+1,x}((Q_{o}\setminus\{0\})\cup Q_{e}\cup X\cup Y,\{0\}))} & \text{if $z < y$}
        \end{cases},
    \end{aligned}
    \end{equation}
    where $X=\begin{cases}
        [z-y]-(a+x+z), & \text{if $y\leq z$}\\
        [y-z]-(a+x+y), & \text{if $z<y$}
    \end{cases}$ and 
    
    $Y=\begin{cases}
        [2a+2x+z+w-1]\setminus[2a+2x+y+z-1]-(a+x+z), & \text{if $y<w$}\\
        [2a+2x+y+z-1]\setminus[2a+2x+z+w-1]-(a+x+z), & \text{if $w\leq y\leq z$}\\
        [2a+2x+2y-1]\setminus[2a+2x+y+w-1]-(a+x+y), & \text{if $z<y$}
    \end{cases}$.
    \end{enumerate}
\end{thm}

Note that the ratios of the tiling generating functions of two $H$ regions (or two $H'$ regions) in \eqref{ece}-\eqref{eci} are given by explicit product formulas due to Theorem \ref{tca} (or Theorem \ref{tcb}). Thus, Theorem \ref{tcd} gives explicit product formulas for the left side of the equations \eqref{ece}-\eqref{eci}. Then combining with the authors' previous work \cite{byunlai2025lozenge}, which provides explicit product formulas of the tiling generating functions $\M_{q}(A(x,y,z,w;a)), \M_{q}(B(x,y,z,w;a)), \M_{q}(C(x,y,z,w;a))$, one can deduce the tiling generating functions of the regions $\M_{q}(A(x,y,z,w;a_1,\ldots,a_{2k+1})),\ldots, \M_{q}(E(x,y,z,w;a_1,\ldots,a_{2k+1}))$ up to a multiplicative factor, which is the ratio of weights of symmetric ferns (see the second to the last paragraph in Section \ref{sec:2}).

Using Theorems \ref{tca} and \ref{tcb}, we can give a simple proof of Theorem \ref{tcd}.

\begin{proof}[Proof of Theorem \ref{tcd}]
    Proof of the five equations \eqref{ece}-\eqref{eci} are almost the same, so we only present the detailed proof of \eqref{ece}. The idea of the proof is the following:

\begin{itemize}
    \item Realize the region (one of the five regions) as $H_{m,n,a,b,c}(L_1, R_1, B)$ or $H'_{m,n,a,b,c}(L_1, R_1, B)$ with suitable choices of the parameters and sets.
    \item Apply Theorem \ref{tca} (if the core of the symmetric fern has odd size) or Theorem \ref{tcb} (if the core of the symmetric fern has even size).
    \item Realize the resulting region as the counterpart region in the statement of Theorem \ref{tcd}.
    \item Keep track of the weight of forced lozenges, then we can deduce Theorem \ref{tcd}.
\end{itemize}

We now apply the above idea to deduce \eqref{ece}. The proof is slightly different depending on whether $y$ satisfies 1) $y < w$, 2) $w\leq y\leq z$, or 3) $z < y$. Thus, we consider three cases separately. 

\textbf{Case 1) $y < w$}

\begin{figure}
    \centering
    \includegraphics[width=0.88\textwidth]{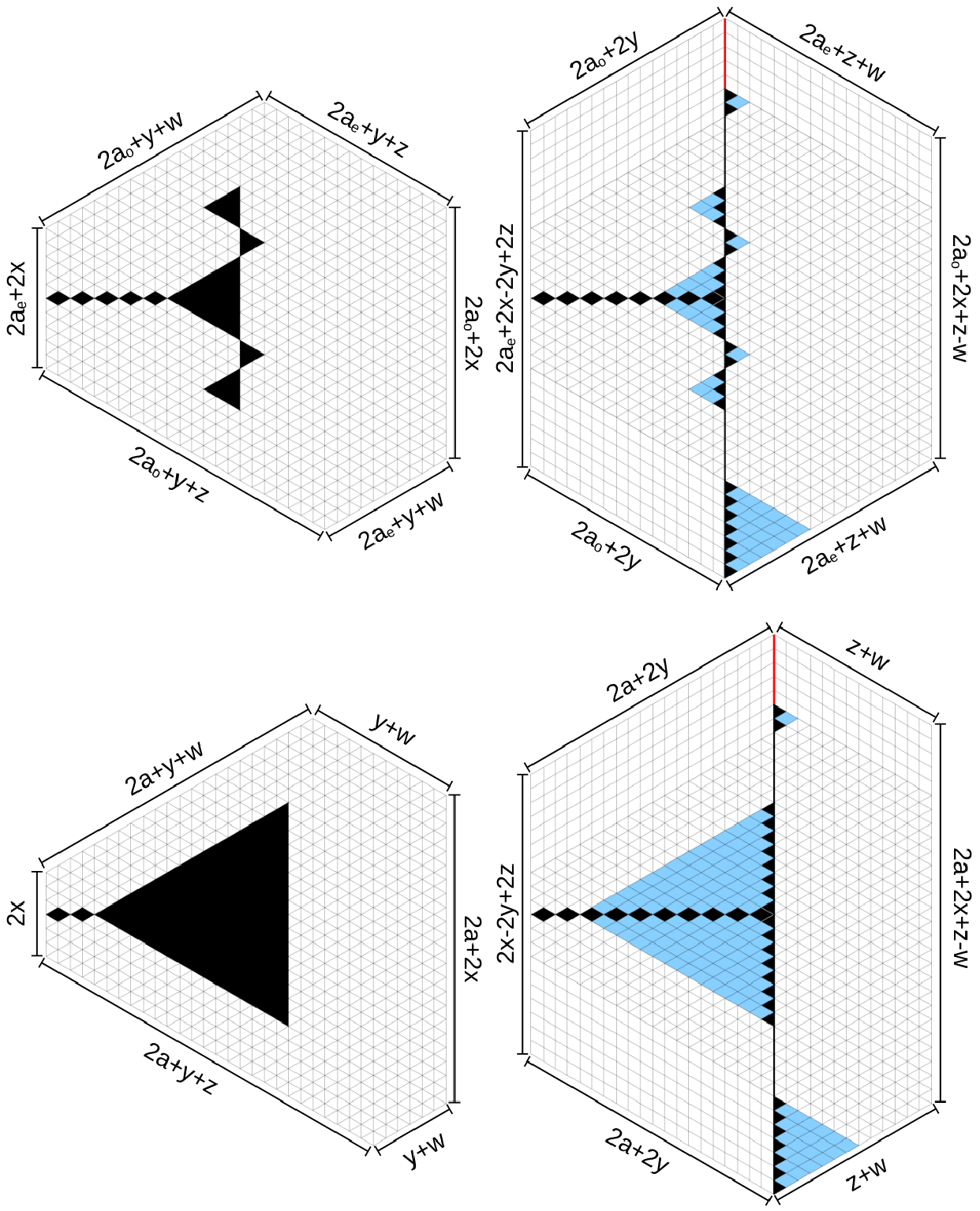}
    \caption{}
    \label{fcc}
\end{figure}

In this case, both vertical lines through either the top or bottom vertex of the boundary hexagon are strictly positioned to the right of the fern axis, as shown in the top left picture in Figure \ref{fcc}. We now consider a region $H'_{0,a_{e},a_{o}+y-1,z+w,x-y+z}(Q_{o},Q_{e}\cup X\cup Y,B)$, where $B=[2a+2x+2z]\setminus[2a+2x+z+w]-\frac{2a+2x+2z+1}{2}$, $X$ and $Y$ are as defined right after \eqref{ece}, and $Q_{o}$ and $Q_{e}$ are as defined in Section 2. We claim that two regions $A(x,y,z,q;a_{1},\ldots,a_{2k+1})$ and $H'_{0,a_{e},a_{o}+y-1,z+w,x-y+z}(Q_{o},Q_{e}\cup X\cup Y,B)$ have the same number of lozenge tilings. This is because, after removing forced lozenges from $H'_{0,a_{e},a_{o}+y-1,z+w,x-y+z}(Q_{o},Q_{e}\cup X\cup Y,B)$, one gets $A(x,y,z,w;a_{1},\ldots,a_{2k+1})$ (compare the two pictures on the top in Figure \ref{fcc}). However, their tiling generating functions are not the same because forced horizontal lozenges have weights different from $1$. Instead, they satisfy the following relation
\begin{equation}\label{ecj}
    M_{q}(A(x,y,z,w;a_{1},\ldots,a_{2k+1}))=\frac{H'_{0,a_{e},a_{o}+y-1,z+w,x-y+z}(Q_{o},Q_{e}\cup X\cup Y,B)}{\wt(F_{sym}(2a_{1},a_{2},\ldots,a_{2k+1}))\cdot \wt(\blacktriangleright_{z-y})\cdot\wt(\blacktriangleright_{w-y})},
\end{equation}
where $\wt(\blacktriangleright_{z-y})$ is the product of weights of all the horizontal lozenges contained in the right-pointing triangle of size $(z-y)$ and $\wt(\blacktriangleright_{w-y})$ is the product of weights of all the horizontal lozenges contained in the right-pointing triangle of size $(w-y)$ as described in Figure \ref{fcc} (in the top right picture, they are collections of shaded lozenges strictly below and above the symmetric fern, respectively). We then flip every right-pointing unit triangle contained in the symmetric fern, which are labeled by elements in $Q_{e}$, following the flipping process described in Theorem \ref{tcb}. The resulting region is $H'_{a_{e},0,a_{o}+y-1,z+w,x-y+z}(Q_{o}\cup Q_{e},X\cup Y,B)$, as can be seen by comparing the two pictures on the right in Figure \ref{fcc}. There are multiple forced lozenges in the new region $H'_{a_{e},0,a_{o}+y-1,z+w,x-y+z}(Q_{o}\cup Q_{e},X\cup Y,B)$ and if we delete them, we get $A(x,y,z,w;a)$ (see the two pictures at the bottom in Figure \ref{fcc}). If we keep track of the weight of every forced lozenge, then the tiling generating functions of these two regions are related by the following relation
\begin{equation}\label{eck}
    M_{q}(A(x,y,z,w;a))=\frac{M_{q}(H'_{a_{e},0,a_{o}+y-1,z+w,x-y+z}(Q_{o}\cup Q_{e},X\cup Y,B))}{\wt(F_{sym}(2a))\cdot \wt(\blacktriangleright_{z-y})\cdot\wt(\blacktriangleright_{w-y})},
\end{equation}
where $\wt(\blacktriangleright_{z-y})$ and $\wt(\blacktriangleright_{w-y})$ are the same as in \eqref{ecj} (compare the positions of the collection of the shaded lozenges strictly below and above the symmetric ferns in the two pictures on the right in Figure \ref{fcc}). Therefore, by taking the ratio of \eqref{ecj} and \eqref{eck}, $\wt(\blacktriangleright_{z-y})$ and $\wt(\blacktriangleright_{w-y})$ cancel out, and we get
\begin{equation}\label{ecl}
\begin{aligned}
    &\frac{\M_{q}(A(x,y,z,w;a_1,\ldots,a_{2k+1}))}{\M_{q}(A(x,y,z,w;a))}\cdot \frac{\wt(F_{sym}(2a_{1},a_{2},\ldots,a_{2k+1}))}{\wt(F_{sym}(2a))}\\
    =&\frac{\M_q(H'_{0,a_{o},a_{e}+y-1,z+w,x-y+z}(Q_{e},Q_{o}\cup X\cup Y,B))}{\M_q(H'_{a_{o},0,a_{e}+y-1,z+w,x-y+z}(Q_{o}\cup Q_{e},X\cup Y,B))}\\
    =&\frac{\M_q(H'_{0,a_{o},a_{e}+y-1,z+w,x-y+z}(Q_{e},Q_{o}\cup X\cup Y))}{\M_q(H'_{a_{o},0,a_{e}+y-1,z+w,x-y+z}(Q_{o}\cup Q_{e},X\cup Y))},
\end{aligned}
\end{equation}
where the second equality is due to Remark \ref{tcc}.

\textbf{Case 2) $w\leq y\leq z$}

\begin{figure}
    \centering
    \includegraphics[width=0.88\textwidth]{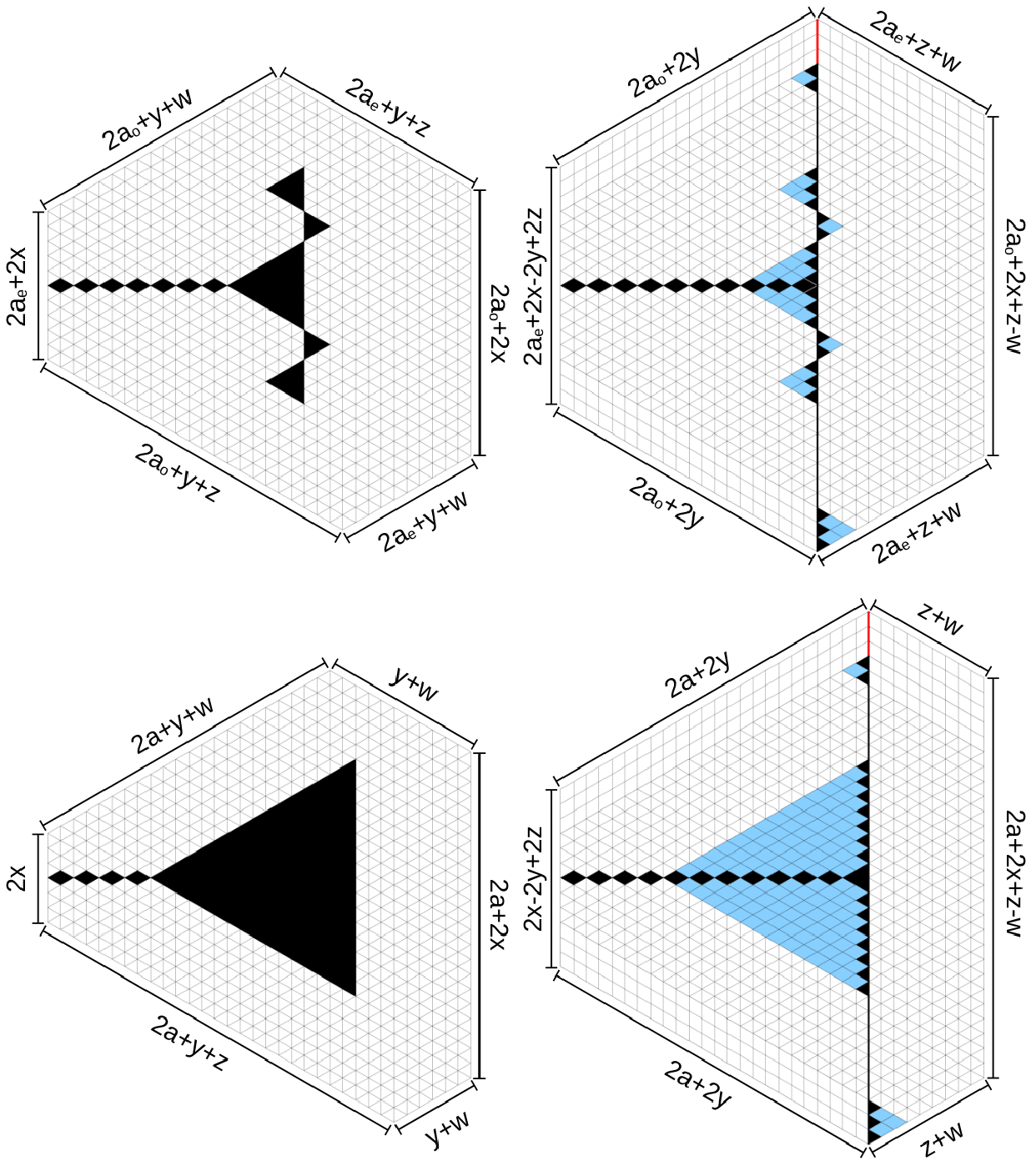}
    \caption{}
    \label{fcd}
\end{figure}

In this case, the fern axis lies between the two vertical axes through either of the top or the bottom vertex of the boundary hexagon (see the top left picture in Figure \ref{fcd}). This case, the region $A(x,y,z,q;a_{1},\ldots,a_{2k+1})$ has the same number of lozenge tilings as $H'_{0,a_{e},a_{o}+y-1,z+w,x-y+z}(Q_{o}\cup Y,Q_{e}\cup X,B)$, where $B=[2a+2x+2z]\setminus[2a+2x+y+z]-\frac{2a+2x+2z+1}{2}$, $X$ and $Y$ are as defined right after \eqref{ece}, and $Q_{o}$ and $Q_{e}$ are as defined in Section 2. This is because if one deletes every forced lozenge from $H'_{0,a_{e},a_{o}+y-1,z+w,x-y+z}(Q_{o}\cup Y,Q_{e}\cup X,B)$, one gets the region $A(x,y,z,w;a_{1},\ldots,a_{2k+1})$ (see the two pictures on the right in Figure \ref{fcd}). However, they do not have the same tiling generating functions again because forced lozenges that have horizontal long diagonals have weights different from $1$. If we keep track of the weight of forced lozenges, the weights of these two regions satisfy the following relation.
\begin{equation}\label{ecm}
    M_{q}(A(x,y,z,w;a_{1},\ldots,a_{2k+1}))=\frac{M_{q}(H'_{0,a_{e},a_{o}+y-1,z+w,x-y+z}(Q_{o}\cup Y,Q_{e}\cup X,B))}{\wt(F_{sym}(2a_{1},a_{2},\ldots,a_{2k+1}))\cdot \wt(\blacktriangleright_{z-y})\cdot\wt(\blacktriangleleft_{y-w})},
\end{equation}
where $\wt(\blacktriangleright_{z-y})$ is the product of weights of all the horizontal lozenges contained in the right-pointing triangle of size $z-y$ and $\wt(\blacktriangleleft_{y-w})$ is the product of weights of all the horizontal lozenges contained in the left-pointing triangle of size $y-w$ (in the top right picture in Figure \ref{fcd}, they are collections of shaded lozenges strictly below and above the symmetric fern, respectively). We then flip every right-pointing unit triangle contained in the symmetric fern, which are labeled by $Q_{e}$, following the flipping process described in Theorem \ref{tcb}. The resulting region is $H'_{a_{e},0,a_{o}+y-1,z+w,x-y+z}(Q_{o}\cup Q_{e}\cup Y,X,B)$, as can be seen by comparing the two pictures on the right in Figure \ref{fcd}. From the new region, if we delete all the forced lozenges, then we get $A(x,y,z,w;a)$ (see the two pictures at the bottom in Figure \ref{fcd}). If we keep track of the weight of every forced lozenge, then the tiling generating functions of these two regions are related by the following equation.
\begin{equation}\label{ecn}
    M_{q}(A(x,y,z,w;a))=\frac{M_{q}(H'_{a_{e},0,a_{o}+y-1,z+w,x-y+z}(Q_{o}\cup Q_{e}\cup Y,X,B))}{\wt(F_{sym}(2a))\cdot \wt(\blacktriangleright_{z-y})\cdot\wt(\blacktriangleleft_{y-w})},
\end{equation}
where $\wt(\blacktriangleright_{z-y})$ and $\wt(\blacktriangleleft_{y-w})$ are the same as in \eqref{ecm} (compare the positions of the shaded lozenges strictly below and above the symmetric ferns in the two pictures on the right in Figure \ref{fcd}). By taking the ratio of \eqref{ecm} and \eqref{ecn}, $\wt(\blacktriangleright_{z-y})$ and $\wt(\blacktriangleleft_{y-w})$ cancel out, and we get the desired identity. Note that, as in \eqref{ecl}, we need to use Remark \ref{tcc} to get rid of $B$.

\textbf{Case 3) $z < y$}

\begin{figure}
    \centering
    \includegraphics[width=0.88\textwidth]{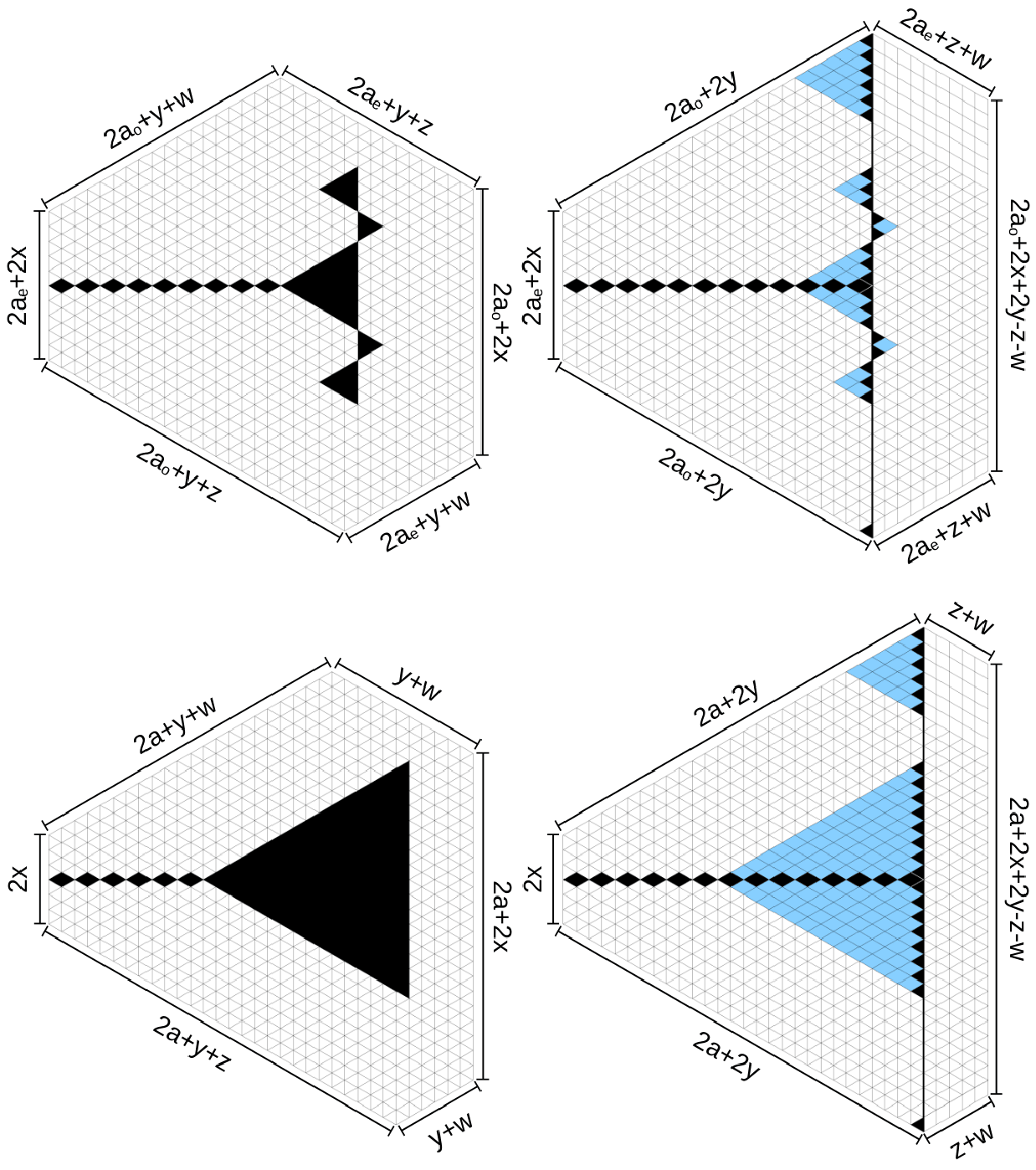}
    \caption{}
    \label{fce}
\end{figure}

In this case, both vertical lines through either of the top or the bottom vertex of the boundary hexagon are positioned strictly left to the fern axis, as shown in the top left picture in Figure \ref{fce}. In this case, we consider the region $H'_{0,a_{e},a_{o}+y-1,z+w,x}(Q_{o}\cup X\cup Y,Q_{e},\varnothing)$, which has the same number of lozenge tilings as $A(x,y,z,w;a_{1},\ldots,a_{2k+1})$ (compare the two pictures on the top in Figure \ref{fce}). Again, they do not have the same tiling generating functions, but instead satisfy the following equation.
\begin{equation}\label{eco}
    M_{q}(A(x,y,z,w;a_{1},\ldots,a_{2k+1}))=\frac{M_{q}(H'_{0,a_{e},a_{o}+y-1,z+w,x}(Q_{o}\cup X\cup Y,Q_{e},\varnothing))}{\wt(F_{sym}(2a_{1},a_{2},\ldots,a_{2k+1}))\cdot \wt(\blacktriangleleft_{y-z})\cdot\wt(\blacktriangleleft_{y-w})},
\end{equation}
where $\wt(\blacktriangleleft_{y-z})$ and $\wt(\blacktriangleleft_{y-w})$ are the products of weights of all horizontal lozenges contained in the left-pointing triangles of size $y-z$ and $y-w$, respectively. (See the top right picture in Figure \ref{fce}. They are collections of shaded lozenges strictly below and above the symmetric fern, respectively. In this example, $y-z=1$ and thus $\wt(\blacktriangleleft_{y-z})=1$.) As in the previous two cases, we flip every right-pointing unit triangle contained in the symmetric fern labeled by $Q_{e}$, following the flipping process described in Theorem \ref{tcb}. The resulting region is $H'_{a_{e},0,a_{o}+y-1,z+w,x}(Q_{o}\cup Q_{e}\cup X\cup Y,\varnothing,\varnothing)$ and this can be seen by comparing the two pictures on the right in Figure \ref{fce}. The new region contains several forced lozenges, and after removing all the forced lozenges, we get the region  $A(x,y,z,w;a)$. The tiling generating functions of these two regions satisfy the following relation due to the weight of forced lozenges.
\begin{equation}\label{ecp}
    M_{q}(A(x,y,z,w;a))=\frac{M_{q}(H'_{a_{e},0,a_{o}+y-1,z+w,x}(Q_{o}\cup Q_{e}\cup X\cup Y,\varnothing,\varnothing))}{\wt(F_{sym}(2a))\cdot \wt(\blacktriangleleft_{y-z})\cdot\wt(\blacktriangleleft_{y-w})},
\end{equation}
where $\wt(\blacktriangleleft_{y-z})$ and $\wt(\blacktriangleleft_{y-w})$ are the same as in the previous equation. Thus, if we take the ratio of \eqref{eco} and \eqref{ecp}, $\wt(\blacktriangleleft_{y-z})$ and $\wt(\blacktriangleleft_{y-w})$ cancel out, and using  Remark \ref{tcc}, we obtain the desired equation.
\end{proof}


\section{Preliminaries: tiling generating functions of some regions}\label{sec:4}

In this section, we gather some preliminary results that will be used in the combined proof of Theorems \ref{tca} and \ref{tcb}.

Let $x$, $y$ be any nonnegative integers and $Z$ be a subset of $[x+y]-\frac{x+y+1}{2}\coloneqq \{-\frac{x+y-1}{2},\dots,\frac{x+y-1}{2}\}$. We label the elements of $Z$ by $z_1, z_2, \dots$ so that $z_i < z_j$ if $i<j$. We then consider a trapezoid region with sides of length $x$, $y$, $x+y$, and $y$ clockwise from the left and denote it by $S_{x,y}$. We label the unit segments on the right side of $S_{x,y}$ by $-\frac{x+y-1}{2},-\frac{x+y-3}{2},\dots,\frac{x+y-3}{2},\frac{x+y-1}{2}$ from the bottom to the top. Now, we delete the left-pointing unit triangles whose labels of their right sides are in $Z$ from $S_{x,y}$ and denote the resulting region by $S_{x,y}(Z)$ (see the left picture in Figure \ref{fda}). We put the region on the $(i,j)$-plane so that the $j$-axis cuts through the horizontal symmetry axis of the boundary trapezoid $S_{x,y}$. Lastly, we assign weight to lozenges in this region: if the center of the horizontal lozenge has $i$ coordinate equal to $n$, we give weight $\frac{q^{n}+q^{-n}}{2}$ to the horizontal lozenge, and we give weight $1$ to every lozenge that is not horizontal. Let $\M_{q}(S_{x,y}(Z))$ be the tiling generating function of the region under this weight. The first lemma states that the tiling generating function $\M_{q}(S_{x,y}(Z))$ is given by a simple product formula.

\begin{figure}
    \centering
    \includegraphics[width=0.88\textwidth]{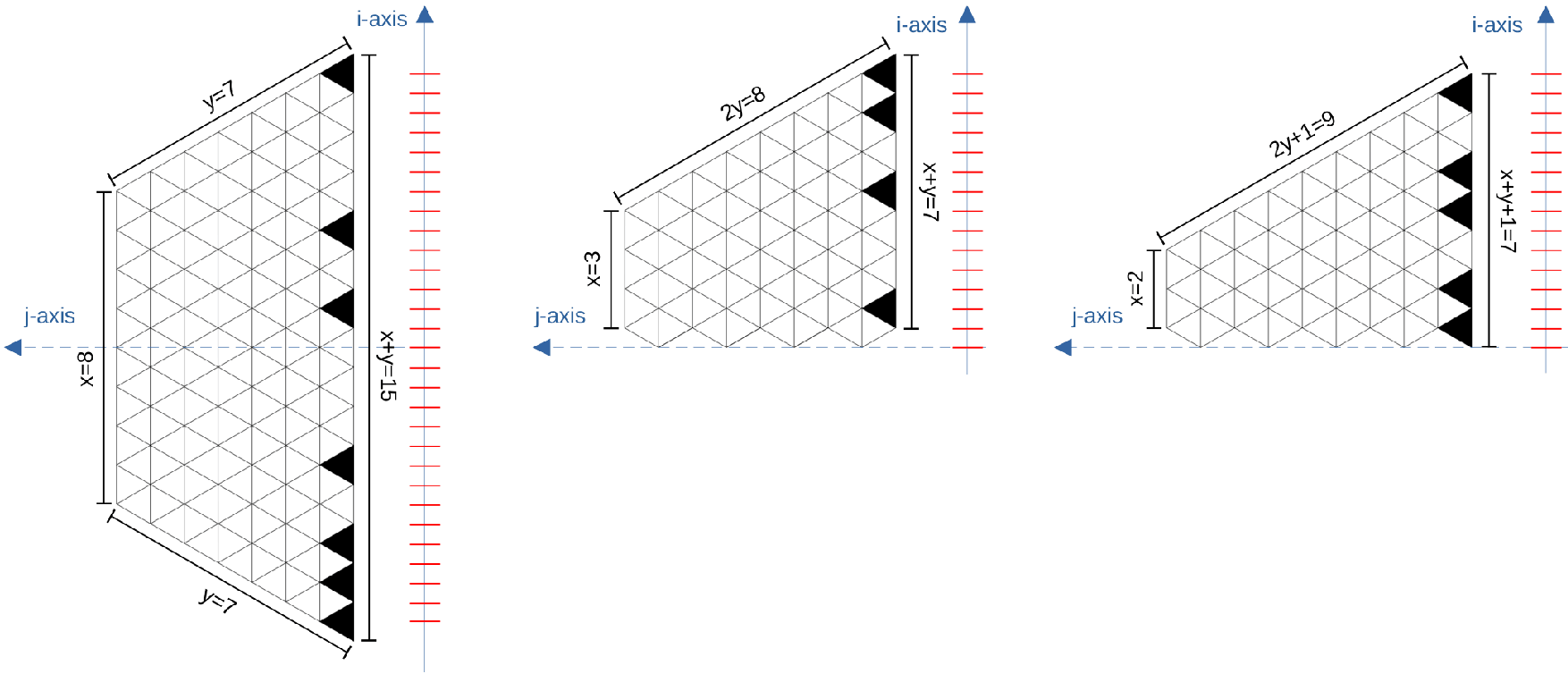}
    \caption{The regions $S_{8,7}(\{-7,-6,-5,-3,1,3,7\})$, $R_{3,8}(\{1,4,6,7\})$, and $R_{2,9}(\{\frac{1}{2},\frac{3}{2},\frac{7}{2},\frac{9}{2},\frac{13}{2}\})$ on the $(i,j)$-planes (from left to right).}
    \label{fda}
\end{figure}

\begin{lemma}\label{tda}
For any nonnegative integers $x$, $y$ and for any set $Z\subseteq [x+y]-\frac{x+y+1}{2}$, $\M_{q}(S_{x,y}(Z))=0$ unless $|Z|=y$. If $|Z|=y$, we have
\begin{equation}\label{eda}
    \M_{q}(S_{x,y}(Z))=\frac{\displaystyle\Delta_{1,1,q}(Z)}{\displaystyle\prod_{i=1}^{y-1}{\langle i\rangle_{q}!}},   
\end{equation}
where\footnote{This notation was already introduced in Section \ref{sec:3}. See the paragraph just before Theorem \ref{tca}.} $\displaystyle\Delta_{1,1,q}(Z)\coloneqq \prod_{z_1, z_2 \in Z, z_1 < z_2}\langle z_2+z_1\rangle_{q}^+\langle z_2-z_1\rangle_{q}$.
\end{lemma}
The proof of Lemma \ref{tda} will be provided in the appendix (see Appendix \ref{sec:App}).

Next, we recall two results of the second author and Rohatgi \cite{LR2021TGF} regarding the tiling generating functions of quartered hexagons with some dents on their right sides. To state the lemma, we define two families of regions $R_{x,2y}(Z)$ and $R_{x,2y+1}(Z)$ as follows.

We first define $R_{x,2y}(Z)$. Let $x$, $y$ be any nonnegative integers and $Z$ be a subset of $[x+y]$. From any point on a triangular lattice, we move $x$ units to the north, $2y$ units to the northeast, and $x+y$ units to the south along the lattice lines. Then we move $2y$ units to the west via a zigzag path, alternating southwest and northwest directions. This closed path determines a region, which is roughly a quarter of a hexagon. We denote this region by $R_{x,2y}$. Then, we label the unit segments on the right side of this region by $1,2,\dots, x+y$ from the bottom to the top. For any subset $Z\subset[x+y]$, $R_{x,2y}(Z)$ is obtained from the region $R_{x,2y}$ by deleting left-pointing unit triangles whose labels of their right sides are in $Z$ (see the middle picture in Figure \ref{fda}). The region $R_{x,2y+1}(Z)$ is defined similarly as follows. From any point on a triangular lattice, we move $x$ units to the north, $2y+1$ units to the northeast, $y+1$ units to the south along the lattice lines. Then we move $2y+1$ units to the west via a zigzag path, alternating northwest and southwest directions. We denote the regions determined by this closed path by $R_{x,2y+1}$. Then we label the unit segments on the right side of this region by $\frac{1}{2},\frac{3}{2},\dots,x+y-\frac{1}{2}, x+y+\frac{1}{2}$ from the bottom to the top. Then, for any subset $Z\subset [x+y+1]-\frac{1}{2}$, the region $R_{x,2y+1}(Z)$ is obtained from $R_{x,2y+1}$ by deleting left-pointing unit triangles whose labels of their right sides are in $Z$ (see the right picture in Figure \ref{fda}). Now, we assign the weights to the lozenges on the regions as follows. We put both regions $R_{x,2y}(Z)$ and $R_{x,2y+1}(Z)$ on the $(i,j)$-plane so that the $j$-axis passes through all the bottommost vertices of the regions. If the center of the horizontal lozenge has $i$-coordinate $n$, then we give a weight $\frac{q^{n}+q^{-n}}{2}$ to the lozenge and give a weight $1$ to every lozenge that is not horizontal. Let $\M_{q}(R_{x,2y}(Z))$ and $\M_{q}(R_{x,2y+1}(Z))$ be the tiling generating functions of the two regions under these weight assignments. Lemma \ref{tdb} provides product formulas for these tiling generating functions. For any finite set $T \subset \mathbb{Z}^{+}$ (or $T \subset \mathbb{Z}^{+}-\frac{1}{2}$), we define $\displaystyle\Delta_{2,1,q}(T)\coloneqq \prod_{t_i, t_j\in T, t_i<t_j}\frac{\langle 2t_{j}+2t_{i}\rangle_{q}}{2}\frac{\langle 2t_{j}-2t_{i}\rangle_{q}}{2}$.

\begin{lemma}\label{tdb}
For any nonnegative integers $x$, $y$ and any set $Z\subseteq [x+y]$, $\M_{q}(R_{x,2y}(Z))=0$ unless $|Z|=y$. If $|Z|=y$, we have
\begin{equation}\label{edb}
    \M_{q}(R_{x,2y}(Z))=\frac{\displaystyle\Delta_{2,1,q}(Z) \cdot \displaystyle\prod_{z\in Z}\frac{\langle 2z\rangle_q}{2}}{\displaystyle\prod_{i=1}^{y}\langle 2i-1\rangle_q!}.
\end{equation}
Similarly, for any nonnegative integers $x$, $y$, and any set $Z\subset [x+y+1]-\frac{1}{2}$, $\M_{q}(R_{x,2y+1}(Z))=0$ unless $|Z|=y+1$. If $|Z|=y+1$, we have
\begin{equation}\label{edc}
    \M_{q}(R_{x,2y+1}(Z))=\frac{\displaystyle\Delta_{2,1,q}(Z)}{\displaystyle\prod_{i=1}^{y+1}\langle 2i-2\rangle_q!}.
\end{equation}
\end{lemma}
The results in this lemma are Theorems 2.1 and 2.2 in \cite{LR2021TGF}. Note that our formulas in \eqref{edb} and \eqref{edc} look different from the formulas in Theorems 2.1 and 2.2 of \cite{LR2021TGF}. This is because two papers are using different $q$-analogue: while we are using $\langle n\rangle_q\coloneqq\frac{q^{n}-q^{-n}}{q-q^{-1}}$, the second author and Rohatgi used $[n]_{q}\coloneqq\frac{1-q^{n}}{1-q}$.

We finish this section with one more lemma, which requires the following notations. For any disjoint finite sets $S,T \subset \mathbb{Z}$ (or $S,T \subset \mathbb{Z}-\frac{1}{2}$), we define $\displaystyle\Delta_{1,2,q}(S,T)$ as follows:
\begin{equation*}
    \displaystyle\Delta_{1,2,q}(S,T)\coloneqq \prod_{s\in S, t\in T, s < t}\langle t+s\rangle_{q}^{+}\langle t-s\rangle_{q}\cdot\prod_{s\in S, t\in T, t < s}\langle s+t\rangle_{q}^{+}\langle s-t\rangle_{q}. 
\end{equation*}

Also, for any disjoint finite sets $S,T \subset \mathbb{Z}^{+}$ (or $S,T \subset \mathbb{Z}^{+}-\frac{1}{2}$), we define $\displaystyle\Delta_{2}(S,T)$ as follows:
\begin{equation*}
    \displaystyle\Delta_{2,2,q}(S,T)\coloneqq \prod_{s\in S, t\in T, s<t}\frac{\langle 2t+2s\rangle_{q}}{2}\frac{\langle 2t-2s\rangle_{q}}{2}\cdot\prod_{s\in S, t\in T, t<s}\frac{\langle 2s+2t\rangle_{q}}{2}\frac{\langle 2s-2t\rangle_{q}}{2}.
\end{equation*}

\begin{lemma}\label{tdc}
The following statements hold.

(1) Let $B, C\subset\mathbb{Z}$ (or $B, C\subset\mathbb{Z}+\frac{1}{2}$) be disjoint finite sets. Then
\begin{equation}\label{edd}
    \Delta_{1,1,q}(B\cup C)=\Delta_{1,1,q}(B)\cdot\Delta_{1,2,q}(B,C)\cdot\Delta_{1,1,q}(C).
\end{equation}

(2) Let $B, C, D\subset\mathbb{Z}$ (or $B, C, D\subset\mathbb{Z}+\frac{1}{2}$) be mutually disjoint finite sets. Then
\begin{equation}\label{ede}
    \Delta_{1,2,q}(B\cup C,D)=\Delta_{1,2,q}(B,D)\cdot\Delta_{1,2,q}(C,D).
\end{equation}

(3) Let $B, C\subset\mathbb{Z}^{+}$ (or $B, C\subset\mathbb{Z}^{+}-\frac{1}{2}$) be disjoint finite sets. Then
\begin{equation}\label{edf}
    \Delta_{2,1,q}(B\cup C)=\Delta_{2,1,q}(B)\cdot\Delta_{2,2,q}(B,C)\cdot\Delta_{2,1,q}(C).
\end{equation}

(4) Let $B, C, D\subset\mathbb{Z}^{+}$ (or $B, C, D\subset\mathbb{Z}^{+}-\frac{1}{2}$) be mutually disjoint finite sets. Then
\begin{equation}\label{edg}
    \Delta_{2,2,q}(B\cup C,D)=\Delta_{2,2,q}(B,D)\cdot\Delta_{2,2,q}(C,D).
\end{equation}

(5) Let $B^{+}\subset\mathbb{Z}^{+}$ (or $B^{+}\subset\mathbb{Z}^{+}-\frac{1}{2}$) be a finite set, and $B\coloneqq B^{+}\cup(-B^{+})$. Then
\begin{equation}\label{edh}
    \Delta_{1,1,q}(B)=[\Delta_{2,1,q}(B^{+})]^2\cdot\prod_{b\in B^{+}} \langle 2b\rangle_{q}. 
\end{equation}

(6) Let $B^{+},C\subset\mathbb{Z}^{+}$ and $R\subset\mathbb{Z}^{-}$ (or $B^{+},C\subset\mathbb{Z}^{+}-\frac{1}{2}$ and $R\subset\mathbb{Z}^{-}+\frac{1}{2}$) be finite sets. Let $B\coloneqq B^{+}\cup(-B^{+})$ and suppose that $C\cup R$ and $B$ are disjoint. Then
\begin{equation}\label{edi}
    \Delta_{1,2,q}(C\cup R, B)=\Delta_{2,2,q}(C, B^{+})\cdot\Delta_{2,2,q}(-R,B^{+}).
\end{equation}
\end{lemma}
The proof of the above lemma is straightforward (one needs to use $\langle n \rangle_{q}^{+}=\langle -n \rangle_{q}^{+}$ and $\langle n \rangle_{q}^{+}\langle n \rangle_{q}=\frac{1}{2}\langle 2n \rangle_{q}^{+}$ for integers $n$), so we omit it.


\section{Proof of Theorems \ref{tca} and \ref{tcb}}\label{sec:5}

In this section, we give a proof of Theorems \ref{tca} and \ref{tcb} using Lemmas \ref{tda}-\ref{tdc}.

\begin{proof}[Proof of Theorems \ref{tca} and \ref{tcb}]
For any set of real numbers $A$, we set $A^{+}\coloneqq A\cap \mathbb{R}^{+}$ and $A^{-}\coloneqq A\cap \mathbb{R}^{-}$. Also, recall that $-A\coloneqq \{-a|a\in A\}$.

We first observe some properties of sets $L_1$, $L_2$, $R_1$, and $R_2$. From the construction of the sets $L_2$ and $R_2$, four sets $L_1$, $L_2$, $R_1$, and $R_2$  satisfy the following conditions: $L_1\cup R_1=L_2\cup R_2, L_1\cap R_1=L_2\cap R_2$. Also, since we flipped unit triangles in a symmetric way, the sets $L_1\setminus L_2 = R_2\setminus R_1$ and $R_1\setminus R_2 = L_2\setminus L_1$ are symmetric, i.e., $L_1\setminus L_2=-(L_1\setminus L_2)$ and $R_1\setminus R_2=-(R_1\setminus R_2)$.

Now, we set $L_{fix}\coloneqq L_1\cap L_2$, $L_{flip}\coloneqq L_1\setminus L_2$, $R_{fix}\coloneqq R_1\cap R_2$, and $R_{flip}\coloneqq R_1\setminus R_2$. Using these newly defined sets, we can express the four sets $L_1$, $L_2$, $R_1$, and $R_2$ as follows:
\begin{equation*}
L_1=L_{fix}\cup L_{flip}, R_1=R_{fix}\cup R_{flip}, L_2=L_{fix}\cup R_{flip}, \text{ and } R_2=R_{fix}\cup L_{flip}.
\end{equation*}
Furthermore, we know that the two sets $L_{flip}$ and $R_{flip}$ are symmetric, i.e. $L_{flip}=-L_{flip}$ and $R_{flip}=-R_{flip}$, and have cardinalities $2m$ and $2n$, respectively. Note that it implies $-L_{flip}^{-}=L_{flip}^{+}$ and $-R_{flip}^{-}=R_{flip}^{+}$.

\begin{figure}
    \centering
    \includegraphics[width=0.88\textwidth]{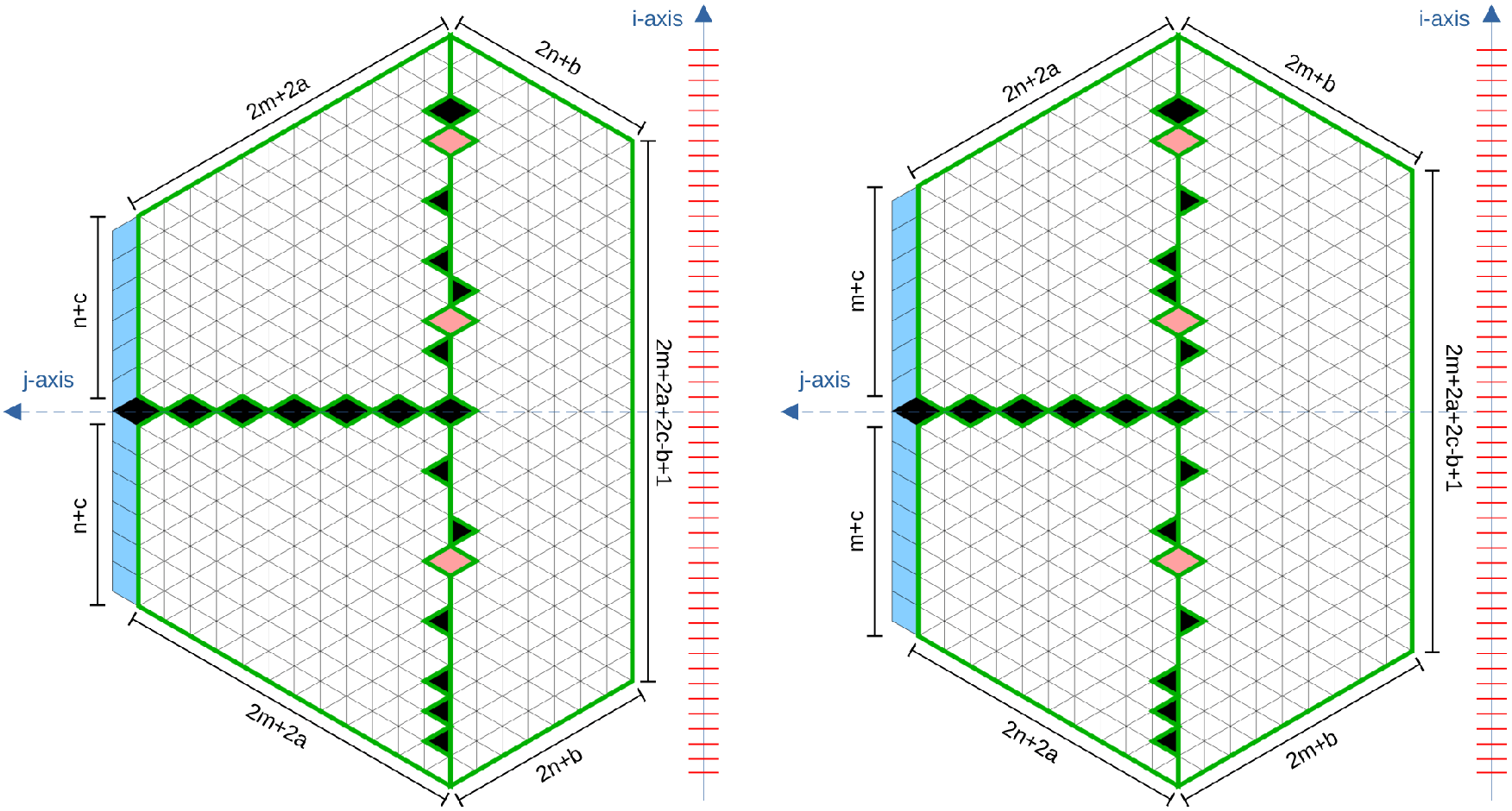}
    \caption{The pictures show how the $H$-regions in Figure \ref{fca} are split into three subregions. In these pictures, $C=\{-5, 3, 9\}$ (see \eqref{eeb} and \eqref{eed}) and the corresponding horizontal lozenges are marked by shaded lozenges across the vertical diagonal.}
    \label{fea}
\end{figure}

We first prove Theorem \ref{tca}. We partition the set of lozenge tilings of the two subregions according to the positions of the horizontal lozenges that are placed across the vertical diagonal. Using this partitioning, the tiling generating function of the region $H_{m,n,a,b,c}(L_1, R_1, B)$ can be expressed as follows.
\begin{equation}\label{eea}
    \M_{q}(H_{m,n,a,b,c}(L_1, R_1, B))=\displaystyle\sum_{C} \Bigg[\prod_{i\in C} \frac{q^{2i}+q^{-2i}}{2}\Bigg] \cdot \M_{q}(H_{m,n,a,b,c}(L_1\cup C, R_1\cup C, B)),
\end{equation}
where $C$ represents the set of labels of all unit segments covered by the horizontal lozenges on the vertical diagonal and the sum in \eqref{eea} runs over all subsets $C$ of $\{-(m+n+a+c), -(m+n+a+c)+1,\dots, m+n+a+c-1, m+n+a+c\}$ such that $|C|=2m+2a-|L_1|$, $C\cap(L_1\cup R_1\cup B)=\varnothing$, and $\M_{q}(H_{m,n,a,b,c}(L_1\cup C, R_1\cup C, B))\neq0$.

Since lozenge tilings of $H_{m,n,a,b,c}(L_1\cup C, R_1\cup C, B)$ cannot have any horizontal lozenges across the vertical diagonal of the region, the tiling generating function $\M_{q}(H_{m,n,a,b,c}(L_1\cup C, R_1\cup C, B))$ can be written as a product of the tiling generating functions of three subregions separated by the vertical diagonal as follows.
\begin{equation}\label{eeb}
\begin{aligned}
    &\M_{q}(H_{m,n,a,b,c}(L_1\cup C, R_1\cup C, B))\\
    &=\M_{q}(R_{n+c, 2m+2a}(-(L_{1}^{-}\cup C^{-})))\cdot\M_{q}(R_{n+c, 2m+2a}(L_{1}^{+}\cup C^{+}))\cdot\M_{q}(S_{2m+2a+2c-b+1, 2n+b}(R_1\cup C)).   
\end{aligned}
\end{equation}
In \eqref{eeb}, we used the fact that the leftmost strips of the regions are uniquely tiled by lozenges weighted by $1$, so we can instead consider the tiling generating functions of the subregions after deleting the leftmost strips (see the left picture in Figure \ref{fea}.)

Similarly, we can also partition the set of lozenge tilings of $H_{n,m,a,b,c}(L_2, R_2, B)$ according to the positions of the horizontal lozenges placed on the vertical diagonal. Thus, we have
\begin{equation}\label{eec}
    \M_{q}(H_{n,m,a,b,c}(L_2, R_2, B))=\displaystyle\sum_{C} \Bigg[\prod_{i\in C} \frac{q^{2i}+q^{-2i}}{2}\Bigg] \cdot \M_{q}(H_{n,m,a,b,c}(L_2\cup C, R_2\cup C, B)),
\end{equation}
where $C$ represents the set labels of all unit segments covered by the horizontal lozenges on the vertical diagonal and the sum in \eqref{eec} runs over all subsets $C$ of $\{-(m+n+a+c), -(m+n+a+c)+1,\dots, m+n+a+c-1, m+n+a+c\}$ such that $|C|=2n+2a-|L_2|$, $C\cap(L_2\cup R_2\cup B)=\varnothing$, and $\M_{q}(H_{n,m,a,b,c}(L_2\cup C, R_2\cup C, B))\neq0$.

Like \eqref{eeb}, the latter term in the summand in \eqref{eec} can also be expressed as a product of the tiling generating functions of three subregions separated by the vertical diagonal as follows.
\begin{equation}\label{eed}
\begin{aligned}
    &\M_{q}(H_{n,m,a,b,c}(L_2\cup C, R_2\cup C, B))\\
    &=\M_{q}(R_{m+c, 2n+2a}(-(L_{2}^{-}\cup C^{-})))\cdot\M_{q}(R_{m+c, 2n+2a}(L_{2}^{+}\cup C^{+}))\cdot\M_{q}(S_{2n+2a+2c-b+1, 2m+b}(R_2\cup C)).    
\end{aligned}
\end{equation}

Note that $2n+2a-|L_2|=2n+2a-(|L_{fix}|+|R_{flip}|)=2n+2a-(|L_{fix}|+2n)=2m+2a-(|L_{fix}|+2m)=2m+2a-(|L_{fix}|+|L_{flip}|)=2m+2a-|L_1|$. Also,
\begin{equation*}
\begin{aligned}
    &\M_{q}(H_{m,n,a,b,c}(L_1\cup C, R_1\cup C, B))\neq0.\\
    \iff&
    \begin{aligned}
        &\M_{q}(R_{n+c, 2m+2a}(-(L_{1}^{-}\cup C^{-})))\neq0, \M_{q}(R_{n+c, 2m+2a}(L_{1}^{+}\cup C^{+}))\neq0,\text{ and }\\
        &\M_{q}(S_{2m+2a+2c-b+1, 2n+b}(R_1\cup C))\neq0.
    \end{aligned}\\
    \iff& |L_{1}^{-}\cup C^{-}|=|L_{1}^{+}\cup C^{+}|=m+a \text{ and } |R_1\cup C|=2n+b.\\
    \iff& |L_{2}^{-}\cup C^{-}|=|L_{2}^{+}\cup C^{+}|=n+a \text{ and } |R_2\cup C|=2m+b.\\
    \iff& 
    \begin{aligned}
        &\M_{q}(R_{m+c, 2n+2a}(-(L_{2}^{-}\cup C^{-})))\neq0, \M_{q}(R_{m+c, 2n+2a}(L_{2}^{+}\cup C^{+}))\neq0,\text{ and }\\
        &\M_{q}(S_{2n+2a+2c-b+1, 2m+b}(R_2\cup C))\neq0.    
    \end{aligned}\\
    \iff& \M_{q}(H_{n,m,a,b,c}(L_2\cup C, R_2\cup C, B))\neq0.
\end{aligned}
\end{equation*}

From this observation, we can conclude that the sums in \eqref{eea} and \eqref{eec} run over the same sets.
Also, the assumption $\M_{q}(H_{m,n,a,b,c}(L_1, R_1, B))\neq 0$ in Theorem \ref{tca} guarantees the existence of $C$ that makes the summands in \eqref{eea} and \eqref{eec} nonzero.
For any such $C$, we will simplify the ratio between the corresponding summands, using Lemmas \ref{tda} and \ref{tdb}. By \eqref{eeb}, \eqref{eed}, and the lemmas,
\begin{equation}\label{eee}
\begin{aligned}
    &\frac{\Bigg[\displaystyle\prod_{i\in C} \frac{q^{2i}+q^{-2i}}{2}\Bigg] \cdot \M_{q}(H_{n,m,a,b,c}(L_2\cup C, R_2\cup C, B))}{\Bigg[\displaystyle\prod_{i\in C} \frac{q^{2i}+q^{-2i}}{2}\Bigg] \cdot \M_{q}(H_{m,n,a,b,c}(L_1\cup C, R_1\cup C, B))}\\
    =&\frac{\M_{q}(R_{m+c, 2n+2a}(-(L_{2}^{-}\cup C^{-})))\cdot\M_{q}(R_{m+c, 2n+2a}(L_{2}^{+}\cup C^{+}))\cdot\M_{q}(S_{2n+2a+2c-b+1, 2m+b}(R_2\cup C))}{\M_{q}(R_{n+c, 2m+2a}(-(L_{1}^{-}\cup C^{-})))\cdot\M_{q}(R_{n+c, 2m+2a}(L_{1}^{+}\cup C^{+}))\cdot\M_{q}(S_{2m+2a+2c-b+1, 2n+b}(R_1\cup C))}\\
    =&\frac{\frac{\displaystyle\Delta_{2,1,q}(-(L_{2}^{-}\cup C^{-})) \cdot \displaystyle\prod_{z\in-(L_{2}^{-}\cup C^{-})}\frac{\langle 2z\rangle_q}{2}}{\displaystyle\prod_{i=1}^{n+a}\langle 2i-1\rangle_q!}\cdot\frac{\displaystyle\Delta_{2,1,q}(L_{2}^{+}\cup C^{+}) \cdot \displaystyle\prod_{z\in L_{2}^{+}\cup C^{+}}\frac{\langle 2z\rangle_q}{2}}{\displaystyle\prod_{i=1}^{n+a}\langle 2i-1\rangle_q!}\cdot\frac{\displaystyle\Delta_{1,1,q}(R_{2}\cup C)}{\displaystyle\prod_{i=1}^{2m+b-1}{\langle i\rangle_{q}!}}}{\frac{\displaystyle\Delta_{2,1,q}(-(L_{1}^{-}\cup C^{-})) \cdot \displaystyle\prod_{z\in-(L_{1}^{-}\cup C^{-})}\frac{\langle 2z\rangle_q}{2}}{\displaystyle\prod_{i=1}^{m+a}\langle 2i-1\rangle_q!}\cdot\frac{\displaystyle\Delta_{2,1,q}(L_{1}^{+}\cup C^{+}) \cdot \displaystyle\prod_{z\in L_{1}^{+}\cup C^{+}}\frac{\langle 2z\rangle_q}{2}}{\displaystyle\prod_{i=1}^{m+a}\langle 2i-1\rangle_q!}\cdot\frac{\displaystyle\Delta_{1,1,q}(R_{1}\cup C)}{\displaystyle\prod_{i=1}^{2n+b-1}{\langle i\rangle_{q}!}}}\\
    =&\frac{\Bigg[\displaystyle\prod_{i=1}^{m+a}\langle 2i-1\rangle_q!\Bigg]^2}{\Bigg[\displaystyle\prod_{i=1}^{n+a}\langle 2i-1\rangle_q!\Bigg]^{2}}\cdot\frac{\displaystyle\prod_{i=1}^{2n+b-1}{\langle i\rangle_{q}!}}{\displaystyle\prod_{i=1}^{2m+b-1}{\langle i\rangle_{q}!}}\cdot\frac{\displaystyle\prod_{z\in L_{2}}\frac{\langle 2|z|\rangle_{q}}{2}}{\displaystyle\prod_{z\in L_{1}}\frac{\langle 2|z|\rangle_{q}}{2}}\cdot\frac{\Delta_{2,1,q}(-(L_{2}^{-}\cup C^{-}))\Delta_{2,1,q}(L_{2}^{+}\cup C^{+})\Delta_{1,1,q}(R_{2}\cup C)}{\Delta_{2,1,q}(-(L_{1}^{-}\cup C^{-}))\Delta_{2,1,q}(L_{1}^{+}\cup C^{+})\Delta_{1,1,q}(R_{1}\cup C)}.
\end{aligned}
\end{equation}

Using the the properties of $\Delta_{1,1,q}$, $\Delta_{1,2,q}$, $\Delta_{2,1,q}$, $\Delta_{2,2,q}$ in Lemma \ref{tdc} and the fact $-L_{flip}^{-}=L_{flip}^{+}$ and $-R_{flip}^{-}=R_{flip}^{+}$, we can further simplify the latter term in \eqref{eee} as follows:

\begin{equation}\label{eef}
\begin{aligned}
    &\frac{\Delta_{2,1,q}(-(L_{2}^{-}\cup C^{-}))\Delta_{2,1,q}(L_{2}^{+}\cup C^{+})\Delta_{1,1,q}(R_{2}\cup C)}{\Delta_{2,1,q}(-(L_{1}^{-}\cup C^{-}))\Delta_{2,1,q}(L_{1}^{+}\cup C^{+})\Delta_{1,1,q}(R_{1}\cup C)}\\
    =&\frac{\Delta_{2,1,q}(-((L^{-}_{fix}\cup R^{-}_{flip})\cup C^{-}))\Delta_{2,1,q}((L^{+}_{fix}\cup R^{+}_{flip})\cup C^{+})\Delta_{1,1,q}((R_{fix}\cup L_{flip})\cup C)}{\Delta_{2,1,q}(-((L^{-}_{fix}\cup L^{-}_{flip})\cup C^{-}))\Delta_{2,1,q}((L^{+}_{fix}\cup L^{+}_{flip})\cup C^{+})\Delta_{1,1,q}((R_{fix}\cup R_{flip})\cup C)}\\
    =&\frac{\Delta_{2,1,q}(-((L^{-}_{fix}\cup C^{-})\cup R^{-}_{flip}))\Delta_{2,1,q}((L^{+}_{fix}\cup C^{+})\cup R^{+}_{flip})\Delta_{1,1,q}((R_{fix}\cup C) \cup L_{flip})}{\Delta_{2,1,q}(-((L^{-}_{fix}\cup C^{-})\cup L^{-}_{flip}))\Delta_{2,1,q}((L^{+}_{fix}\cup C^{+})\cup L^{+}_{flip})\Delta_{1,1,q}((R_{fix}\cup C)\cup R_{flip})}\\
    =&\frac{\Delta_{2,1,q}(-(L^{-}_{fix}\cup C^{-}))\Delta_{2,2,q}(-L^{-}_{fix}, -R^{-}_{flip})\Delta_{2,2,q}(-C^{-}, -R^{-}_{flip})\Delta_{2,1,q}(-R^{-}_{flip})}{\Delta_{2,1,q}(-(L^{-}_{fix}\cup C^{-}))\Delta_{2,2,q}(-L^{-}_{fix}, -L^{-}_{flip})\Delta_{2,2,q}(-C^{-}, -L^{-}_{flip})\Delta_{2,1,q}(-L^{-}_{flip})}\\
    &\cdot\frac{\Delta_{2,1,q}(L^{+}_{fix}\cup C^{+})\Delta_{2,2,q}(L^{+}_{fix}, R^{+}_{flip})\Delta_{2,2,q}(C^{+}, R^{+}_{flip})\Delta_{2,1,q}(R^{+}_{flip})}{\Delta_{2,1,q}(L^{+}_{fix}\cup C^{+})\Delta_{2,2,q}(L^{+}_{fix}, L^{+}_{flip})\Delta_{2,2,q}(C^{+}, L^{+}_{flip})\Delta_{2,1,q}(L^{+}_{flip})}\\
    &\cdot\frac{\Delta_{1,1,q}(R_{fix}\cup C)\Delta_{1,2,q}(R_{fix}, L_{flip})\Delta_{1,2,q}(C, L_{flip})\Delta_{1,1,q}(L_{flip})}{\Delta_{1,1,q}(R_{fix}\cup C)\Delta_{1,2,q}(R_{fix}, R_{flip})\Delta_{1,2,q}(C, R_{flip})\Delta_{1,1,q}(R_{flip})}\\
    =&
    \frac{\Delta_{1,1,q}(L_{flip})}{\Delta_{2,1,q}(L^{+}_{flip})\Delta_{2,1,q}(-L^{-}_{flip})}\cdot\frac{\Delta_{2,1,q}(R^{+}_{flip})\Delta_{2,1,q}(-R^{-}_{flip})}{\Delta_{1,1,q}(R_{flip})}\\
    &\cdot\frac{\Delta_{2,2,q}(-L^{-}_{fix}, -R^{-}_{flip})\Delta_{2,2,q}(L^{+}_{fix}, R^{+}_{flip})}{\Delta_{2,2,q}(-L^{-}_{fix}, -L^{-}_{flip})\Delta_{2,2,q}(L^{+}_{fix}, L^{+}_{flip})}\cdot\frac{\Delta_{1,2,q}(R_{fix}, L_{flip})}{\Delta_{1,2,q}(R_{fix}, R_{flip})}\\
    &\cdot\frac{\Delta_{2,2,q}(-C^{-}, -R^{-}_{flip})\Delta_{2,2,q}(C^{+}, R^{+}_{flip})}{\Delta_{2,2,q}(-C^{-}, -L^{-}_{flip})\Delta_{2,2,q}(C^{+}, L^{+}_{flip})}\cdot\frac{\Delta_{1,2,q}(C, L_{flip})}{\Delta_{1,2,q}(C, R_{flip})}\\
    =&\frac{\displaystyle\prod_{s\in L^{+}_{flip}}\langle 2s\rangle_q}{\displaystyle\prod_{s\in R^{+}_{flip}}\langle 2s\rangle_q}\cdot\frac{\Delta_{1,2,q}(L_{fix}, R_{flip})}{\Delta_{1,2,q}(L_{fix}, L_{flip})}\cdot\frac{\Delta_{1,2,q}(R_{fix}, L_{flip})}{\Delta_{1,2,q}(R_{fix}, R_{flip})}\cdot\frac{\Delta_{1,2,q}(C, R_{flip})}{\Delta_{1,2,q}(C, L_{flip})}\cdot\frac{\Delta_{1,2,q}(C, L_{flip})}{\Delta_{1,2,q}(C, R_{flip})}\\
    =&\frac{\displaystyle\prod_{s\in L^{+}_{flip}}\langle 2s\rangle_q}{\displaystyle\prod_{s\in R^{+}_{flip}}\langle 2s\rangle_q}\cdot\frac{\Delta_{1,1,q}(L_{fix})\Delta_{1,2,1}(L_{fix}, R_{flip})\Delta_{1,1,q}(R_{flip})}{\Delta_{1,1,q}(L_{fix})\Delta_{1,2,q}(L_{fix}, L_{flip})\Delta_{1,1,q}(L_{flip})}\cdot\frac{\Delta_{1,1,q}(R_{fix})\Delta_{1,2,q}(R_{fix}, L_{flip})\Delta_{1,1,q}(L_{flip})}{\Delta_{1,1,q}(R_{fix})\Delta_{1,2,q}(R_{fix}, R_{flip})\Delta_{1,1,q}(R_{flip})}\\
    =&\frac{\displaystyle\prod_{s\in L^{+}_{flip}}\langle 2s\rangle_q}{\displaystyle\prod_{s\in R^{+}_{flip}}\langle 2s\rangle_q}\cdot\frac{\Delta_{1,1,q}(L_{fix}\cup R_{flip})}{\Delta_{1,1,q}(L_{fix}\cup L_{flip})}\cdot\frac{\Delta_{1,1,q}(R_{fix}\cup L_{flip})}{\Delta_{1,1,q}(R_{fix}\cup R_{flip})}\\
    =&\frac{\displaystyle\prod_{s\in L^{+}_{flip}}\langle 2s\rangle_q}{\displaystyle\prod_{s\in R^{+}_{flip}}\langle 2s\rangle_q}\cdot\frac{\Delta_{1,1,q}(L_{2})\Delta_{1,1,q}(R_{2})}{\Delta_{1,1,q}(L_{1})\Delta_{1,1,q}(R_{1})}.    
\end{aligned}
\end{equation}

Since 
\begin{equation}\label{eeg}
\begin{aligned}
    \frac{\displaystyle\prod_{z\in L_{2}}\frac{\langle 2|z|\rangle_{q}}{2}}{\displaystyle\prod_{z\in L_{1}}\cdot\frac{\langle 2|z|\rangle_{q}}{2}}\cdot\frac{\displaystyle\prod_{s\in L^{+}_{flip}}\langle 2s\rangle_q}{\displaystyle\prod_{s\in R^{+}_{flip}}\langle 2s\rangle_q}
    =\frac{\displaystyle\prod_{z\in L_{fix}}\frac{\langle 2|z|\rangle_{q}}{2}\prod_{z\in R_{flip}}\frac{\langle 2|z|\rangle_{q}}{2}}{\displaystyle\prod_{z\in L_{fix}}\frac{\langle 2|z|\rangle_{q}}{2}\prod_{z\in L_{flip}}\frac{\langle 2|z|\rangle_{q}}{2}}\cdot\frac{\displaystyle\prod_{s\in L^{+}_{flip}}\langle 2s\rangle_q}{\displaystyle\prod_{s\in R^{+}_{flip}}\langle 2s\rangle_q}
    =&\frac{\displaystyle\prod_{z\in R_{flip}}\frac{\langle 2|z|\rangle_{q}}{2}}{\displaystyle\prod_{z\in L_{flip}}\frac{\langle 2|z|\rangle_{q}}{2}}\cdot\frac{\displaystyle\prod_{s\in L^{+}_{flip}}\langle 2s\rangle_q}{\displaystyle\prod_{s\in R^{+}_{flip}}\langle 2s\rangle_q}\\
    =&\frac{\displaystyle\prod_{z\in R_{flip}}\frac{\langle 2|z|\rangle_{q}}{2}}{\displaystyle\prod_{z\in L_{flip}}\frac{\langle 2|z|\rangle_{q}}{2}}\cdot\frac{\displaystyle\prod_{s\in L^{+}_{flip}}\langle 2s\rangle_q}{\displaystyle\prod_{s\in R^{+}_{flip}}\langle 2s\rangle_q}\\
    =&\frac{\displaystyle\prod_{z\in R_{flip}^{+}}\frac{\langle 2|z|\rangle_{q}}{4}}{\displaystyle\prod_{z\in L_{flip}^{+}}\frac{\langle 2|z|\rangle_{q}}{4}}\\
    =&\sqrt{\frac{\displaystyle\prod_{z\in R_{flip}}\frac{\langle 2|z|\rangle_{q}}{4}}{\displaystyle\prod_{z\in L_{flip}}\frac{\langle 2|z|\rangle_{q}}{4}}}\\
    =&\sqrt{\frac{\displaystyle\prod_{z\in L_{2}}\frac{\langle 2|z|\rangle_{q}}{4}}{\displaystyle\prod_{z\in L_{1}}\frac{\langle 2|z|\rangle_{q}}{4}}},
\end{aligned}
\end{equation}
combining \eqref{eee}, \eqref{eef}, and \eqref{eeg} we get
\begin{equation}\label{eeh}
\begin{aligned}
    &\frac{\Bigg[\displaystyle\prod_{i\in C} \frac{q^{2i}+q^{-2i}}{2}\Bigg] \cdot \M_{q}(H_{n,m,a,b,c}(L_2\cup C, R_2\cup C, B))}{\Bigg[\displaystyle\prod_{i\in C} \frac{q^{2i}+q^{-2i}}{2}\Bigg] \cdot \M_{q}(H_{m,n,a,b,c}(L_1\cup C, R_1\cup C, B))}\\
    =&\frac{\Bigg[\displaystyle\prod_{i=1}^{m+a}\langle 2i-1\rangle_q!\Bigg]^2}{\Bigg[\displaystyle\prod_{i=1}^{n+a}\langle 2i-1\rangle_q!\Bigg]^{2}}\cdot\frac{\displaystyle\prod_{i=1}^{2n+b-1}{\langle i\rangle_{q}!}}{\displaystyle\prod_{i=1}^{2m+b-1}{\langle i\rangle_{q}!}}\cdot\frac{\displaystyle\prod_{z\in L_{2}}\frac{\langle 2|z|\rangle_{q}}{2}}{\displaystyle\prod_{z\in L_{1}}\cdot\frac{\langle 2|z|\rangle_{q}}{2}}\cdot\frac{\displaystyle\prod_{s\in L^{+}_{flip}}\langle 2s\rangle_q}{\displaystyle\prod_{s\in R^{+}_{flip}}\langle 2s\rangle_q}\cdot\frac{\Delta_{1,1,q}(L_{2})\Delta_{1,1,q}(R_{2})}{\Delta_{1,1,q}(L_{1})\Delta_{1,1,q}(R_{1})}\\
    =&\frac{\Bigg[\displaystyle\prod_{i=1}^{m+a}\langle 2i-1\rangle_q!\Bigg]^2}{\Bigg[\displaystyle\prod_{i=1}^{n+a}\langle 2i-1\rangle_q!\Bigg]^{2}}\cdot\frac{\displaystyle\prod_{i=1}^{2n+b-1}{\langle i\rangle_{q}!}}{\displaystyle\prod_{i=1}^{2m+b-1}{\langle i\rangle_{q}!}}\cdot\sqrt{\frac{\displaystyle\prod_{z\in L_{2}}\frac{\langle 2|z|\rangle_{q}}{4}}{\displaystyle\prod_{z\in L_{1}}\frac{\langle 2|z|\rangle_{q}}{4}}}\cdot\frac{\Delta_{1,1,q}(L_{2})\Delta_{1,1,q}(R_{2})}{\Delta_{1,1,q}(L_{1})\Delta_{1,1,q}(R_{1})}.
\end{aligned}
\end{equation}

Since the expression on the right does not depend on a choice of the set $C$, by \eqref{eea}, \eqref{eec}, and \eqref{eeh}, we have

\begin{equation}\label{eei}
\begin{aligned}
    \frac{\M_{q}(H_{n,m,a,b,c}(L_2, R_2, B))}{\M_{q}(H_{m,n,a,b,c}(L_1, R_1, B))}&=\frac{\Bigg[\displaystyle\prod_{i=1}^{m+a}\langle 2i-1\rangle_q!\Bigg]^2}{\Bigg[\displaystyle\prod_{i=1}^{n+a}\langle 2i-1\rangle_q!\Bigg]^{2}}\cdot\frac{\displaystyle\prod_{i=1}^{2n+b-1}{\langle i\rangle_{q}!}}{\displaystyle\prod_{i=1}^{2m+b-1}{\langle i\rangle_{q}!}}\cdot\sqrt{\frac{\displaystyle\prod_{z\in L_{2}}\frac{\langle 2|z|\rangle_{q}}{4}}{\displaystyle\prod_{z\in L_{1}}\frac{\langle 2|z|\rangle_{q}}{4}}}\cdot\frac{\Delta_{1,1,q}(L_{2})\Delta_{1,1,q}(R_{2})}{\Delta_{1,1,q}(L_{1})\Delta_{1,1,q}(R_{1})}
\end{aligned}
\end{equation}
and this completes the proof of Theorem \ref{tca}. 

The proof of Theorem \ref{tcb} is very similar to that of Theorem \ref{tca}.
Following the same partitioning as in the proof of Theorem \ref{tca}, we express the tiling generating function of the region $H'_{m,n,a,b,c}(L_1, R_1, B)$ as a sum as follows:

\begin{equation}\label{eej}
    \M_{q}(H'_{m,n,a,b,c}(L_1, R_1, B))=\displaystyle\sum_{C} \Bigg[\prod_{i\in C} \frac{q^{2i}+q^{-2i}}{2}\Bigg] \cdot \M_{q}(H'_{m,n,a,b,c}(L_1\cup C, R_1\cup C, B)).
\end{equation}
where $C$ represents the set of labels of all unit segments covered by the horizontal lozenges on the vertical diagonal and the sum in \eqref{eej} runs over all subsets $C$ of $\{-(m+n+a+c)-\frac{1}{2}, -(m+n+a+c)+\frac{1}{2},\dots, m+n+a+c-\frac{1}{2}, m+n+a+c+\frac{1}{2}\}$ such that $|C|=2m+2a+2-|L_1|$, $C\cap(L_1\cup R_1\cup B)=\varnothing$, and $\M_{q}(H'_{m,n,a,b,c}(L_1\cup C, R_1\cup C, B))\neq0$.

\begin{figure}
    \centering
\includegraphics[width=0.9\textwidth]{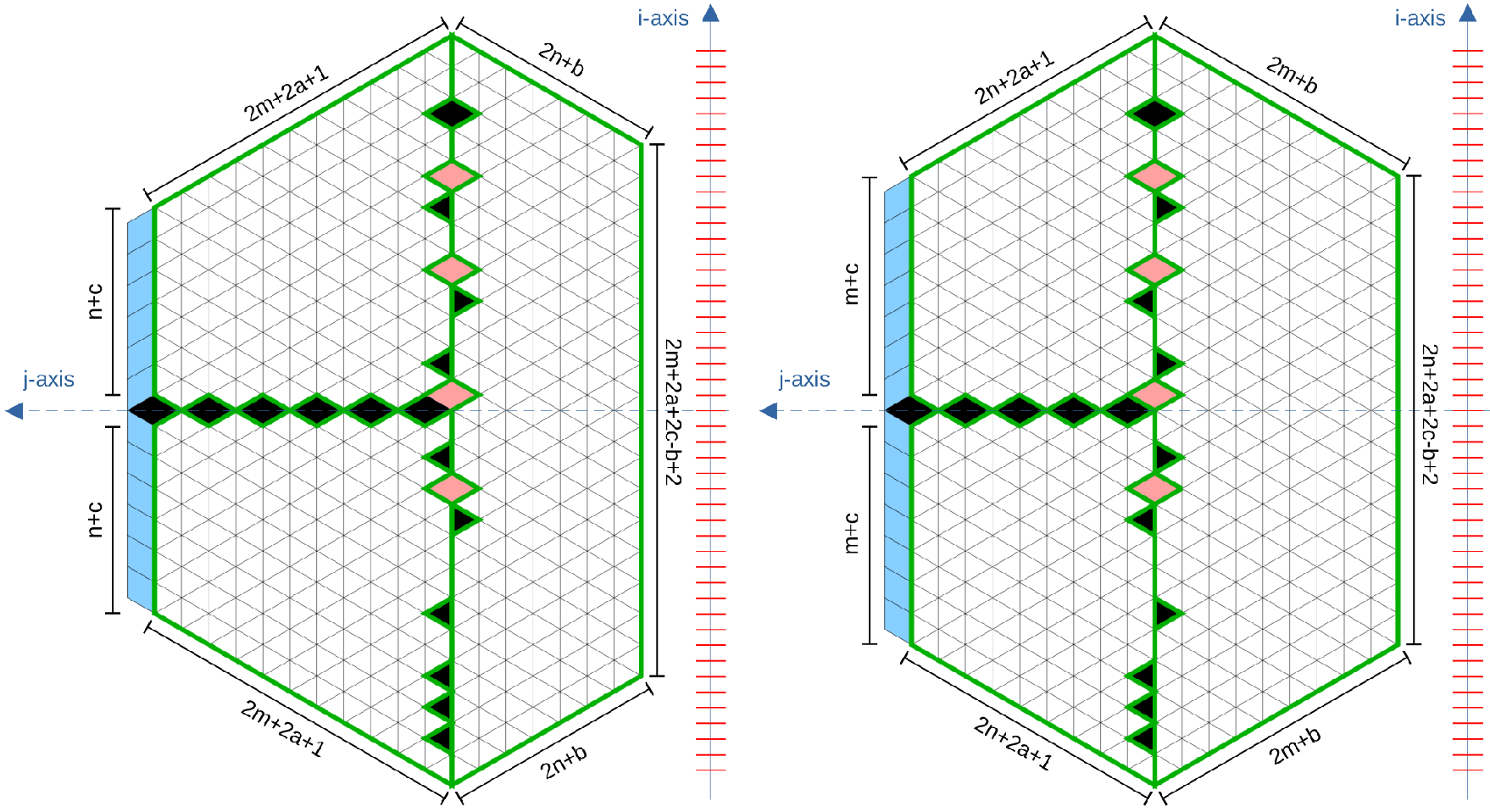}
    \caption{The pictures show how the $H'$-regions in Figure \ref{fcb} are split into three subregions. In these pictures, $C=\{-\frac{5}{2}, \frac{1}{2}, \frac{9}{2}, \frac{15}{2}\}$ (see \eqref{eek}) and the corresponding horizontal lozenges are marked by shaded lozenges across the vertical diagonals.}
    \label{feb}
\end{figure}

Again, since lozenge tilings of $H'_{m,n,a,b,c}(L_1\cup C, R_1\cup C, B)$ cannot have any horizontal lozenge across the vertical diagonal, its tiling generating function is the product of the tiling generating functions of three subregions separated by the vertical diagonal. Thus, we have
\begin{equation}\label{eek}
\begin{aligned}
    &\M_{q}(H'_{m,n,a,b,c}(L_1\cup C, R_1\cup C, B))\\
    =&\M_{q}(R_{n+c, 2m+2a+1}(-(L_{1}^{-}\cup C^{-})))\cdot\M_{q}(R_{n+c, 2m+2a+1}(L_{1}^{+}\cup C^{+}))\\
    &\cdot\M_{q}(S_{2m+2a+2c-b+2, 2n+b}(R_1\cup C)).    
\end{aligned}
\end{equation}

Applying the same argument to $H'_{n,m,a,b,c}(L_2, R_2, B)$, we have
\begin{equation}\label{eel}
    \M_{q}(H'_{n,m,a,b,c}(L_2, R_2, B))=\displaystyle\sum_{C} \Bigg[\prod_{i\in C} \frac{q^{2i}+q^{-2i}}{2}\Bigg] \cdot \M_{q}(H'_{n,m,a,b,c}(L_2\cup C, R_2\cup C, B)),
\end{equation}
and
\begin{equation}\label{eem}
\begin{aligned}
    &\M_{q}(H'_{n,m,a,b,c}(L_2\cup C, R_2\cup C, B))\\
    =&\M_{q}(R_{m+c,2n+2a+1}(-(L_{2}^{-}\cup C^{-})))\cdot\M_{q}(R_{m+c,2n+2a+1}(L_{2}^{+}\cup C^{+}))\\
    &\cdot\M_{q}(S_{2n+2a+2c-b+2, 2m+b}(R_2\cup C)).    
\end{aligned}
\end{equation}

One can easily check that \eqref{eej} and \eqref{eel} are summed over the same sets. Also, the assumption $\M_{q}(H'_{m,n,a,b,c}(L_1, R_1, B))\neq 0$ in Theorem \ref{tcb} guarantees the existence of $C$ that makes the summands in \eqref{eej} and \eqref{eel} nonzero. Now, we observe the ratio of corresponding summands in \eqref{eej} and \eqref{eel}. For any $C$ that makes both summands in \eqref{eej} and \eqref{eel} nonzero, by \eqref{eek}, \eqref{eem}, and Lemmas \ref{tda} and \ref{tdb},
\begin{equation}\label{een}
\begin{aligned}
    &\frac{\Bigg[\displaystyle\prod_{i\in C} \frac{q^{2i}+q^{-2i}}{2}\Bigg] \cdot \M_{q}(H'_{n,m,a,b,c}(L_2\cup C, R_2\cup C, B))}{\Bigg[\displaystyle\prod_{i\in C} \frac{q^{2i}+q^{-2i}}{2}\Bigg] \cdot \M_{q}(H'_{m,n,a,b,c}(L_1\cup C, R_1\cup C, B))}\\
    =&\frac{\M_{q}(R_{m+c,2n+2a+1}(-(L_{2}^{-}\cup C^{-})))\cdot\M_{q}(R_{m+c,2n+2a+1}(L_{2}^{+}\cup C^{+}))\cdot\M_{q}(S_{2n+2a+2c-b+2, 2m+b}(R_2\cup C))}{\M_{q}(R_{n+c, 2m+2a+1}(-(L_{1}^{-}\cup C^{-})))\cdot\M_{q}(R_{n+c, 2m+2a+1}(L_{1}^{+}\cup C^{+}))\cdot\M_{q}(S_{2m+2a+2c-b+2, 2n+b}(R_1\cup C))}\\
    =&\frac{\frac{\displaystyle\Delta_{2,1,q}(-(L_{2}^{-}\cup C^{-}))}{\displaystyle\prod_{i=1}^{n+a+1}\langle 2i-2\rangle_q!}\cdot\frac{\displaystyle\Delta_{2,1,q}(L_{2}^{+}\cup C^{+})}{\displaystyle\prod_{i=1}^{n+a+1}\langle 2i-2\rangle_q!}\cdot\frac{\displaystyle\Delta_{1,1,q}(R_2\cup C)}{\displaystyle\prod_{i=1}^{2m+b-1}\langle i\rangle_q!}}{\frac{\displaystyle\Delta_{2,1,q}(-(L_{1}^{-}\cup C^{-}))}{\displaystyle\prod_{i=1}^{m+a+1}\langle 2i-2\rangle_q!}\cdot\frac{\displaystyle\Delta_{2,1,q}(L_{1}^{+}\cup C^{+})}{\displaystyle\prod_{i=1}^{m+a+1}\langle 2i-2\rangle_q!}\cdot\frac{\displaystyle\Delta_{1,1,q}(R_1\cup C)}{\displaystyle\prod_{i=1}^{2n+b-1}{\langle i\rangle_{q}!}}}\\
    =&\frac{\Bigg[\displaystyle\prod_{i=1}^{m+a+1}\langle 2i-2\rangle_q!\Bigg]^2}{\Bigg[\displaystyle\prod_{i=1}^{n+a+1}\langle 2i-2\rangle_q!\Bigg]^2}\cdot\frac{\displaystyle\prod_{i=1}^{2n+b-1}{\langle i\rangle_{q}!}}{\displaystyle\prod_{i=1}^{2m+b-1}{\langle i\rangle_{q}!}}\cdot\frac{\Delta_{2,1,q}(-(L_{2}^{-}\cup C^{-}))\Delta_{2,1,q}(L_{2}^{+}\cup C^{+})\Delta_{1,1,q}(R_{2}\cup C)}{\Delta_{2,1,q}(-(L_{1}^{-}\cup C^{-}))\Delta_{2,1,q}(L_{1}^{+}\cup C^{+})\Delta_{1,1,q}(R_{1}\cup C)}.
\end{aligned}
\end{equation}

In \eqref{eef}, we showed that the latter term in \eqref{een} does not depend on the set $C$. Thus, combining \eqref{eej}, \eqref{eel}, \eqref{een}, and \eqref{eef}, we have
\begin{equation}\label{eeo}
\begin{aligned}
    &\frac{\M_{q}(H'_{n,m,a,b,c}(L_2\cup C, R_2\cup C, B))}{\M_{q}(H'_{m,n,a,b,c}(L_1\cup C, R_1\cup C, B))}\\
    =&\frac{\Bigg[\displaystyle\prod_{i=1}^{m+a+1}\langle 2i-2\rangle_q!\Bigg]^2}{\Bigg[\displaystyle\prod_{i=1}^{n+a+1}\langle 2i-2\rangle_q!\Bigg]^2}\cdot\frac{\displaystyle\prod_{i=1}^{2n+b-1}{\langle i\rangle_{q}!}}{\displaystyle\prod_{i=1}^{2m+b-1}{\langle i\rangle_{q}!}}\cdot\frac{\displaystyle\prod_{s\in L^{+}_{flip}}\langle 2s\rangle_q}{\displaystyle\prod_{s\in R^{+}_{flip}}\langle 2s\rangle_q}\cdot\frac{\Delta_{1,1,q}(L_{2})\Delta_{1,1,q}(R_{2})}{\Delta_{1,1,q}(L_{1})\Delta_{1,1,q}(R_{1})}\\
    =&\frac{\Bigg[\displaystyle\prod_{i=1}^{m+a+1}\langle 2i-2\rangle_q!\Bigg]^2}{\Bigg[\displaystyle\prod_{i=1}^{n+a+1}\langle 2i-2\rangle_q!\Bigg]^2}\cdot\frac{\displaystyle\prod_{i=1}^{2n+b-1}{\langle i\rangle_{q}!}}{\displaystyle\prod_{i=1}^{2m+b-1}{\langle i\rangle_{q}!}}\cdot\sqrt{\frac{\displaystyle\prod_{s\in L_{flip}}\langle 2|s|\rangle_q}{\displaystyle\prod_{s\in R_{flip}}\langle 2|s|\rangle_q}}\cdot\frac{\Delta_{1,1,q}(L_{2})\Delta_{1,1,q}(R_{2})}{\Delta_{1,1,q}(L_{1})\Delta_{1,1,q}(R_{1})}\\
    =&\frac{\Bigg[\displaystyle\prod_{i=1}^{m+a+1}\langle 2i-2\rangle_q!\Bigg]^2}{\Bigg[\displaystyle\prod_{i=1}^{n+a+1}\langle 2i-2\rangle_q!\Bigg]^2}\cdot\frac{\displaystyle\prod_{i=1}^{2n+b-1}{\langle i\rangle_{q}!}}{\displaystyle\prod_{i=1}^{2m+b-1}{\langle i\rangle_{q}!}}\cdot\sqrt{\frac{\displaystyle\prod_{s\in L_{1}}\langle 2|s|\rangle_q}{\displaystyle\prod_{s\in L_{2}}\langle 2|s|\rangle_q}}\cdot\frac{\Delta_{1,1,q}(L_{2})\Delta_{1,1,q}(R_{2})}{\Delta_{1,1,q}(L_{1})\Delta_{1,1,q}(R_{1})}.
\end{aligned}
\end{equation}
This completes the proof of Theorem \ref{tcb}. 
\end{proof}




\bibliography{bibliography}{}
\bibliographystyle{abbrv}

\newpage
\appendix
\section{Proof of Lemma \ref{tda}}\label{sec:App}
The goal of this section is to give a proof of Lemma \ref{tda}. This is achieved by proving a more general theorem (Theorem \ref{tfa}) and specializing it.

\begin{figure}
    \centering
    \includegraphics[width=0.8\textwidth]{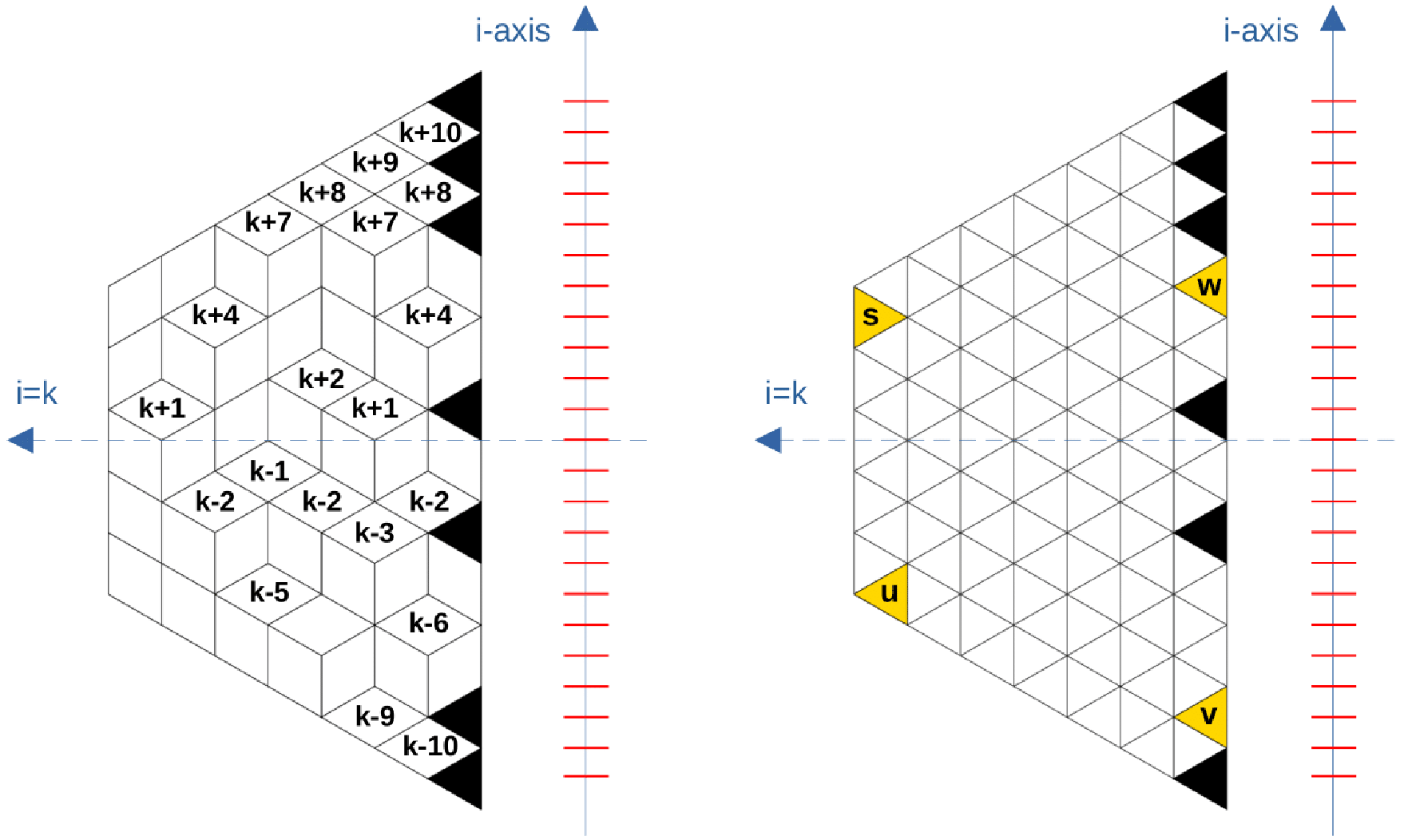}
    \caption{The picture on the left is $S_{5,7}(\overline{\{1,2,5,7,10,11,12\}})$ on $(i,j)$-coordinate plane. The picture on the right is obtained from the left picture by filling in the second dent from the bottom and denoted by $R$ in the proof of Theorem \ref{tfa}.}
    \label{ffa}
\end{figure}

We consider the trapezoidal region $S_{x,y}$ with sides of length $x$, $y$, $x+y$, and $y$ clockwise from the left (this region was introduced in Section \ref{sec:4}). We label the unit segments on the right side of $S_{x,y}$ by $-\frac{x+y-1}{2},-\frac{x+y-3}{2},\dots,\frac{x+y-3}{2},\frac{x+y-1}{2}$ from the bottom to the top. Then for any subset $W\subseteq[x+y]$, let $S_{x,y}(\overline{W})$ be the region obtained from $S_{x,y}$ by deleting left-pointing unit triangles whose right sides are labeled by elements in $W-\frac{x+y+1}{2}=\{w-\frac{x+y+1}{2}|w\in W\}$. Note that this notation is not consistent with the one introduced in Section \ref{sec:4}: the region $S_{x,y}(\overline{W})$ is the same region as $S_{x,y}(W-\frac{x+y+1}{2})$ of Section \ref{sec:4}. We are using a different notation here because the proof of Theorem \ref{tfa} is cleaner with this new notation.

We then place this region $S_{x,y}(\overline{W})$ on the $(i,j)$-coordinate plane so that 1) the line $i=k$ is the perpendicular bisector of the left side of the region and 2) the half of the side length of unit triangles is the unit length of the $i$-axis. (See the left picture in Figure \ref{ffa}; ignore the labels on horizontal lozenges at this point.) We then assign weights to all lozenges bounded by $S_{x,y}(\overline{W})$ as follows. We give a weight $1$ to every lozenge whose long diagonal is not horizontal. For horizontal lozenges, if the center of the lozenge has $i$-coordinate $n$, then we assign a weight $\frac{Xq^{n}+Yq^{-n}}{2}$ (see the left picture in Figure \ref{ffa}). Let $\M_{(q,X,Y,k)}(S_{x,y}(\overline{W}))$ be the tiling generating function of $S_{x,y}(\overline{W})$ under this weight assignment. The following theorem gives a simple product formula for $\M_{(q,X,Y,k)}(S_{x,y}(\overline{W}))$.

\begin{thm}\label{tfa}
For any nonnegative integers $x$, $y$ and for any set $W\subseteq [x+y]$, $S_{x,y}(\overline{W})$ has no lozenge tiling unless $|W|=y$. If $|W|=y$, we have
\begin{equation}\label{efa}
    \M_{(q,X,Y,k)}(S_{x,y}(\overline{W}))=\frac{\displaystyle\Delta_{1,1,(q,X,Y,k)}(W)}{\displaystyle\prod_{i=1}^{y-1}{\langle i\rangle_{q}!}}
\end{equation}
where $\displaystyle\Delta_{1,1,(q,X,Y,k)}(W)\coloneqq \prod_{w_1, w_2 \in W, w_1 < w_2}\Bigg[\frac{q^{(w_1+w_2+k)-(x+y+1)}X+q^{-(w_1+w_2+k)-(x+y+1)}Y}{2}\Bigg]\langle w_2-w_1\rangle_{q}$.
\end{thm}
The second author proved the special case of Theorem \ref{tfa} in \cite{lai2022ratio} using a version of Kuo's graphical condensation \cite{kuo2004applications}. Since we follow the same proof, we recall the version of Kuo's graphical condensation. Let $G$ be a bipartite graph and $V_{1}$ and $V_{2}$ be the two partitions of the set of vertices of $G$. We also denote the set of edges of $G$ by $E$ and denote $G=(V_{1}, V_{2}, E)$. We assume that there is a weight function defined on $E$ (i.e., every edge is weighted). For any graph $G$, a \textit{perfect matching} of $G$ is a collection of edges of $G$ that covers each vertex of $G$ exactly once. Given a perfect matching, its \textit{weight} is the product of the weights of all edges that constitute the perfect matching. Then, the \textit{matching generating function} of $G$ is the sum of weights of all perfect matchings of $G$, and we denote it by $\M(G)$. For any set of vertices $W$ of $G$, $G-W$ is the subgraph of $G$ that is obtained from $G$ by deleting all vertices in $W$ and edges adjacent to at least one vertex in $W$.

\begin{thm}[\cite{kuo2004applications}, Theorem 5.3]
Let $G=(V_1, V_2, E)$ be a weighted plane bipartite graph in which $|V_1|=|V_2|+1$. Let vertices $u,v,w,$ and $s$ appear in a cyclic order on a face of $G$. If $u,v,w \in V_1$ and $s \in V_2$, then
\begin{equation}\label{efb}
    \M(G-\{v\})\M(G-\{u,w,s\})=\M(G-\{u\})\M(G-\{v,w,s\})+\M(G-\{w\})\M(G-\{u,v,s\}).
\end{equation}
\label{tfb}
\end{thm}

For any region $R$ on a triangular lattice, its \textit{dual graph} is a graph such that 1) its vertices are the unit triangles in $R$ and 2) two vertices are joined by an edge if the corresponding unit triangles share an edge in $R$. If every unit lozenge in $R$ is weighted, the dual graph has an induced weight on its edges: each edge is weighted by the weight of the corresponding unit lozenge in $R$. By its construction, unit lozenges in $R$ are in one-to-one correspondence with the edges of the dual graph, and this correspondence induces a natural bijection between lozenge tilings of $R$ and the perfect matchings of its dual graph. Furthermore, this bijection is weight-preserving, and it allows us to prove Theorem \ref{tfa} using Theorem \ref{tfb}.

\begin{proof}[Proof of Theorem \ref{tfa}]
    Note that the trapezoid region $S_{x,y}$ consists of $y$ more left-pointing unit triangles than right-pointing unit triangles. Since every lozenge consists of one left-pointing and one right-pointing unit triangles, unless $|W|=y$, the region $S_{x,y}(\overline{W})$ has no tilings. In the rest of the proof, we assume $|W|=y$ and write $W=\{w_1,\ldots,w_y\}$ where elements are written in increasing order.

    We first show that we only have to prove this theorem when $w_1=1$ and $w_y=x+y$. Suppose $(w_1,w_y)\neq(1,x+y)$ and assume that $(w_1,w_y)=(1+a,(x+y)-b)$ for some integers $a$ and $b$ such that $1\leq 1+a<(x+y)-b\leq x+y$. As shown in left picture in Figure \ref{ffb}, one can check that the tiling generating function $\M_{(q,X,Y,k)}(S_{x,y}(\overline{W}))$ is the same as the tiling generating function of its subregion $\M_{(q,X,Y,k-b+a)}(S_{x-a-b,y}(\overline{W-a}))$. Furthermore, one can check that Equation \eqref{efa} yields the same product formula for $\M_{(q,X,Y,k-b+a)}(S_{x-a-b,y}(\overline{W-a}))$ and $\M_{(q,X,Y,k)}(S_{x,y}(\overline{W}))$. Therefore, it suffices to prove the theorem when $(w_1,w_y)=(1,x+y)$.

\begin{figure}
    \centering
    \includegraphics[width=0.9\textwidth]{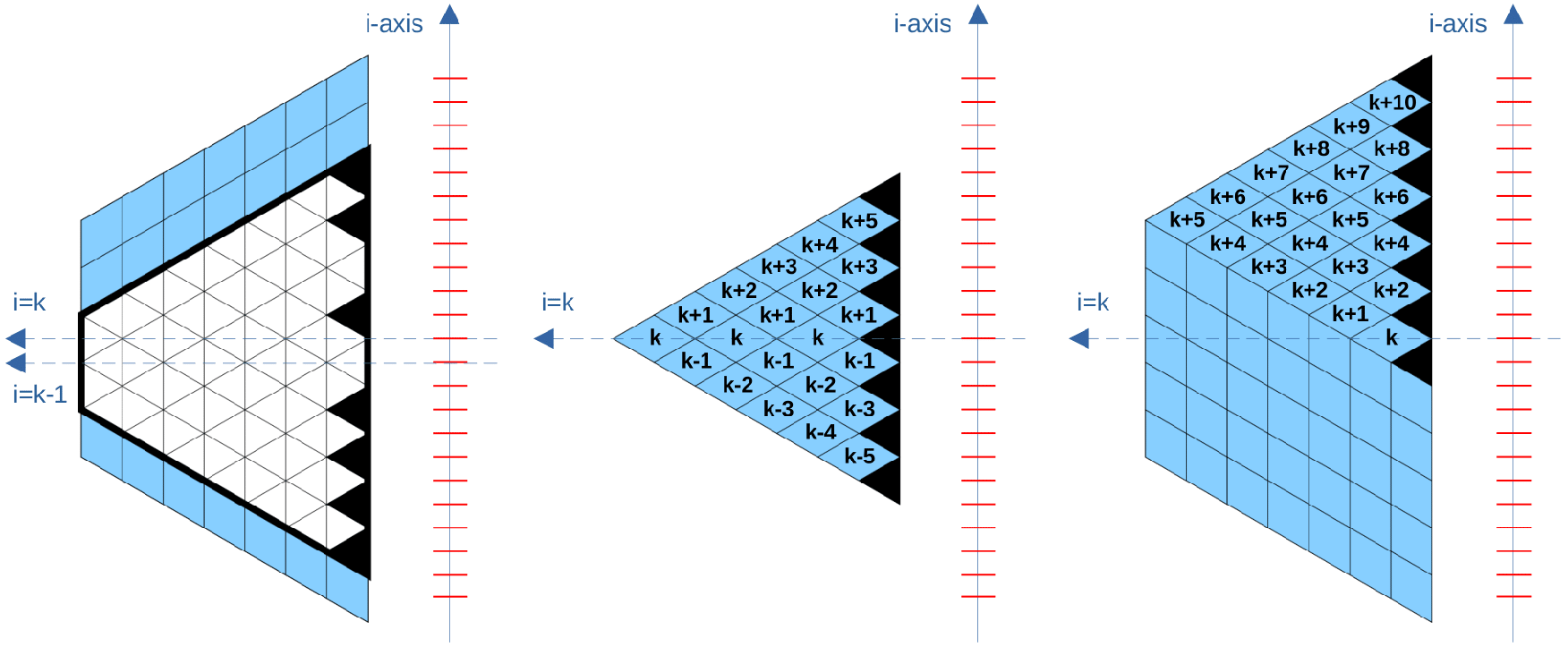}
    \caption{The left picture illustrates why $\M_{q,X,Y,k}(S_{5,7}(2,3,4,5,7,9,10))$ (tiling generating function of the entire region) and $\M_{q,X,Y,k-1}(S_{2,7}(1,2,3,4,6,8,9))$ (tiling generating function of the subregion bounded by thick lines) are the same. The middle picture shows the unique tiling of $S_{0,7}((1,2,3,4,5,6,7))$. The right picture shows the unique tiling of $S_{5,7}(6,7,8,9,10,11,12)$.}
    \label{ffb}
\end{figure}

    Given a region $S_{x,y}(\overline{W})$, let $l$ be the size of the maximal dent cluster attached to the top-right corner of the trapezoid region $S_{x,y}(\overline{W})$. More precisely, $l$ is the unique integer satisfying $w_{y-i}=x+y-i$ for $0\leq i<l$ and  $w_{y-l}<x+y-l$ (thus, $x+y-l\notin W$). For example, $l=3$ for the region in the left picture in \ref{ffa}. We then set $t=y-l$. We will prove the theorem using an induction on $x+y+t$.

    The base cases will be the cases when one of the three parameters $x$, $y$, and $t$ is equal to $0$. We first show that the theorem holds in these cases.
    \begin{itemize}
        \item If $x=0$, then $W=[y]=\{1,\ldots,y\}$ and in this case, the region has a unique tiling that consists of horizontal lozenges only (see the middle picture in Figure \ref{ffb}). In this case, one can check that the tiling generating function of this region, which is the weight of the unique tiling, matches what Equation \eqref{efa} yields.
        \item If $y=0$, then $W=\varnothing$ and the region becomes degenerate, so it has only one tiling (empty tiling) whose weight is $1$. In this case, the formula in Equation \eqref{efa} also yields $1$ (because the empty product is $1$), so the theorem holds.
        \item If $t=0$, then the region also has only one tiling (see the right picture in Figure \ref{ffb}) and one can check that the weight of the tiling matches what Equation \eqref{efa} yields.
    \end{itemize}

    Therefore, we can assume that all three parameters $x$, $y$, and $t$ are positive. As mentioned earlier, we will construct a recurrence relation using Kuo's graphical condensation. To do that, we consider the region $S_{x,y}(\overline{W})$ and let $k$ be the smallest integer such that $w_k=k$ but $w_{k+1}>k+1$ (such $k$ exists because $w_1=1$ and $w_y=x+y>y$). Let $R$ be the region obtained from $S_{x,y}(\overline{W})$ by filling in the $k$th dent from the bottom whose right side is labeled by $k-\frac{x+y+1}{2}$. Then, we choose four unit triangles on $R$ as described in the right picture in Figure \ref{ffa}. More precisely,
    \begin{itemize}
        \item $u$ is the lowest left-pointing unit triangle from the leftmost strip,
        \item $v$ is the position of the dent that we filled in from $S_{x,y}(\overline{W})$ whose right side is labeled by $k-\frac{x+y+1}{2}$,
        \item $w$ is the unit triangle whose right side is labeled by $x+y-l-\frac{x+y+1}{2}$ (this is the dent position right below the maximal dent cluster at the top-right corner),
        \item $s$ is the highest right-pointing unit triangle from the leftmost strip.
    \end{itemize}
    
    By the assumption that $x,y,t>0$, these four unit triangles are distinct and well-defined, thus can be chosen. We then apply Kuo's graphical condensation to the dual graph\footnote{See the paragraph right after the statement of Theorem \ref{tfb}.} of $R$ with four vertices corresponding to the unit triangles $u$, $v$, $w$, and $s$. If we keep track of the weights of forced lozenges, we obtain the following recurrence relation (see Figure \ref{ffc} that illustrates the six terms in \eqref{efc}).
    \begin{figure}
        \centering
        \includegraphics[width=1\textwidth]{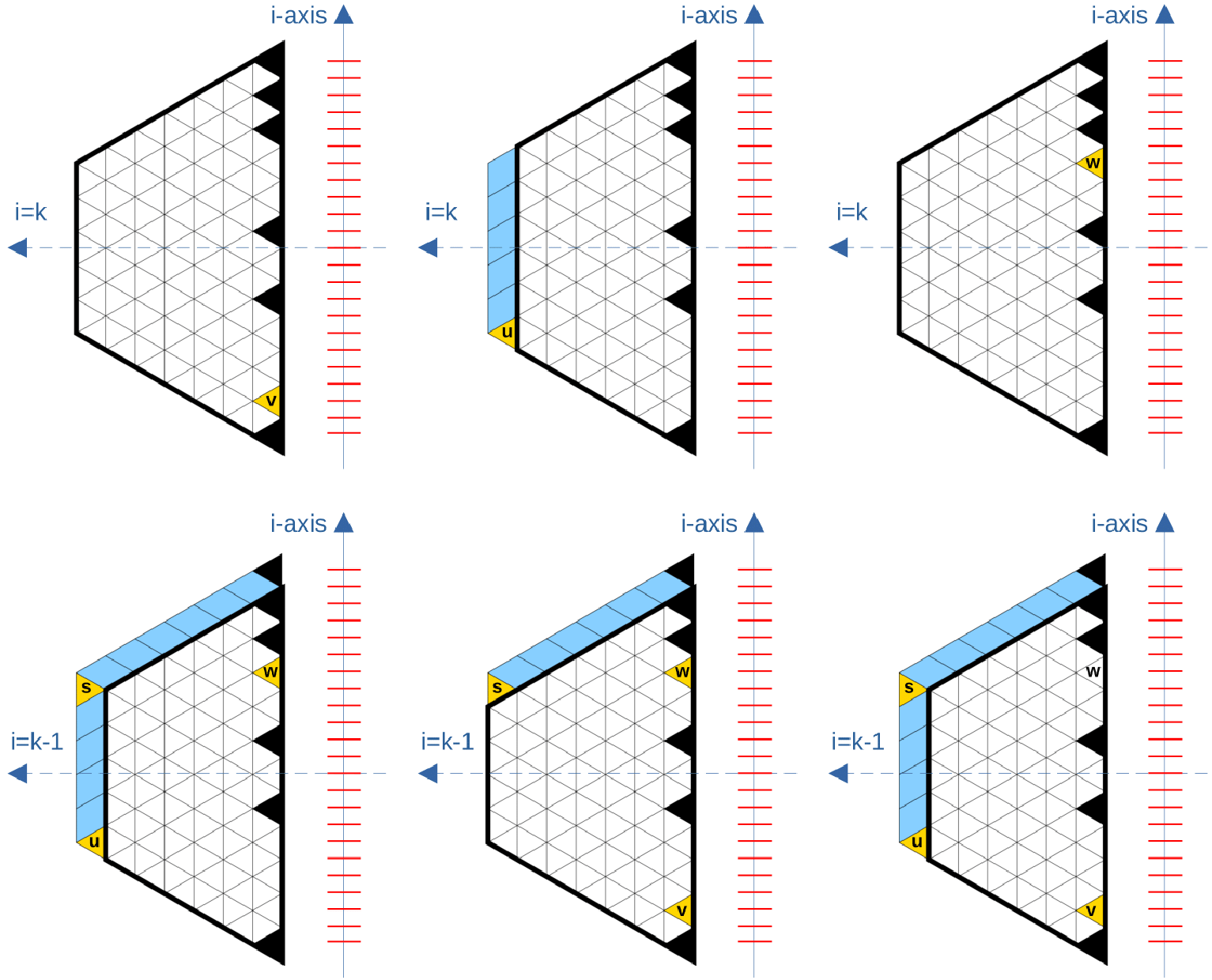}
        \caption{Illustration of the six regions appearing in \eqref{efc}.}
        \label{ffc}
    \end{figure}    
    \begin{equation}\label{efc}
    \begin{aligned}
        &\M_{(q,X,Y,k)}(S_{x,y}(\overline{W}))\cdot\M_{(q,X,Y,k-1)}(S_{x,y-1}(\overline{W\cup\{x+y-l\}\setminus\{k,x+y\}}))\\
        =&\M_{(q,X,Y,k)}(S_{x+1,y-1}(\overline{W\setminus\{k\}}))\cdot\M_{(q,X,Y,k-1)}(S_{x-1,y}(\overline{W\cup\{x+y-l\}\setminus\{x+y\}}))\\
        &+\M_{(q,X,Y,k)}(S_{x,y}(\overline{W\cup\{x+y-l\}\setminus\{k\}}))\cdot\M_{(q,X,Y,k-1)}(S_{x,y-1}(\overline{W\setminus\{x+y\}})).
    \end{aligned}
    \end{equation}
    (We canceled out the weights of forced lozenges.) One can check that the formula in Equation \eqref{efa} satisfies the above recurrence relation. Furthermore, the parameter $x+y+t$ of the six regions appearing in \eqref{efc} is $x+y+(y-l)=x+2y-l$, $x+(y-1)+(y-1-l)=x-2y-l-2$, $(x+1)+(y-1)+(y-1-l)=x+2y-l-1$, $(x-1)+y+(y-l)=x+2y-l-1$, $x+y+(y-(l+1))=x+2y-l-1$, and $x+(y-1)+(y-1-(l-1))=x+2y-l-1$, so the proof of the theorem follows from the induction.
\end{proof}

As a direct consequence of Theorem \ref{tfa}, we can give a simple proof of Theorem \ref{tda}.
\begin{proof}[Proof of Theorem \ref{tda}]
    Not that Theorem \ref{tda} is a special case of Theorem \ref{tfa} when $X=Y=1$, $k=0$ and $W=Z+\frac{x+y+1}{2}=\{z+\frac{x+y+1}{2}~|~z\in Z\}$. By plugging in these values in Equation \eqref{efa}, we get the formula in \eqref{eda}.
\end{proof}

As one more application of Theorem \ref{tfa}, we deduce the tiling generating function of a hexagonal region. For any nonnegative integers $a$, $b$, and $c$, consider a hexagon with sides of length $a$, $b$, $c$, $a$, $b$, and $c$ clockwise from the left side and denote it by $H_{a,b,c}$. We put $H_{a,b,c}$ on the $(i,j)$-coordinate system\footnote{We assume that the unit length of the $i$-axis is the half of the side length of unit triangles.} so that the line $i=k$ is the perpendicular bisector of the left side of $H_{a,b,c}$ for some fixed integer $k$ (see the left picture in Figure \ref{ffd}). We then assign weight to all lozenges bounded by $H_{a,b,c}$. We give weight $1$ to every lozenge whose long diagonal is not horizontal. For horizontal lozenges, if the center of the lozenge has $i$-coordinate $n$, then we assign a weight $\frac{Xq^{n}+Yq^{-n}}{2}$. Let us denote the tiling generating function of $H_{a,b,c}$ under this weight assignment by $\M_{(q,X,Y,k)}(H_{a,b,c})$. Note that the special case of the following corollary when $X=Y=1$ and $k=0$ was presented and used in the authors' previous work \cite[Lemma 2.1]{byunlai2025lozenge}.

\begin{cor}\label{tfc}
    For any nonnegative integers $a$, $b$, and $c$, the tiling generating function of $H_{a,b,c}$ under the weight assignment described above is given by the following formula.
    \begin{equation}\label{efd}
        \M_{(q,X,Y,k)}(H_{a,b,c})=\Bigg[\prod_{i=1}^{b}\prod_{j=1}^{c}\frac{Xq^{k+i-j}+Yq^{-(k+i-j)}}{2}\Bigg]\cdot\frac{H_{q}(a)H_{q}(b)H_{q}(c)H_{q}(a+b+c)}{H_{q}(a+b)H_{q}(b+c)H_{q}(c+a)}.
    \end{equation}
    where $H_{q}(n)\coloneqq\prod_{i=0}^{n-1}\langle i\rangle_{q}!$ for positive integers $n$ and $H_{q}(0)\coloneqq1$.
\end{cor}
\begin{proof}
    \begin{figure}
        \centering
        \includegraphics[width=0.8\textwidth]{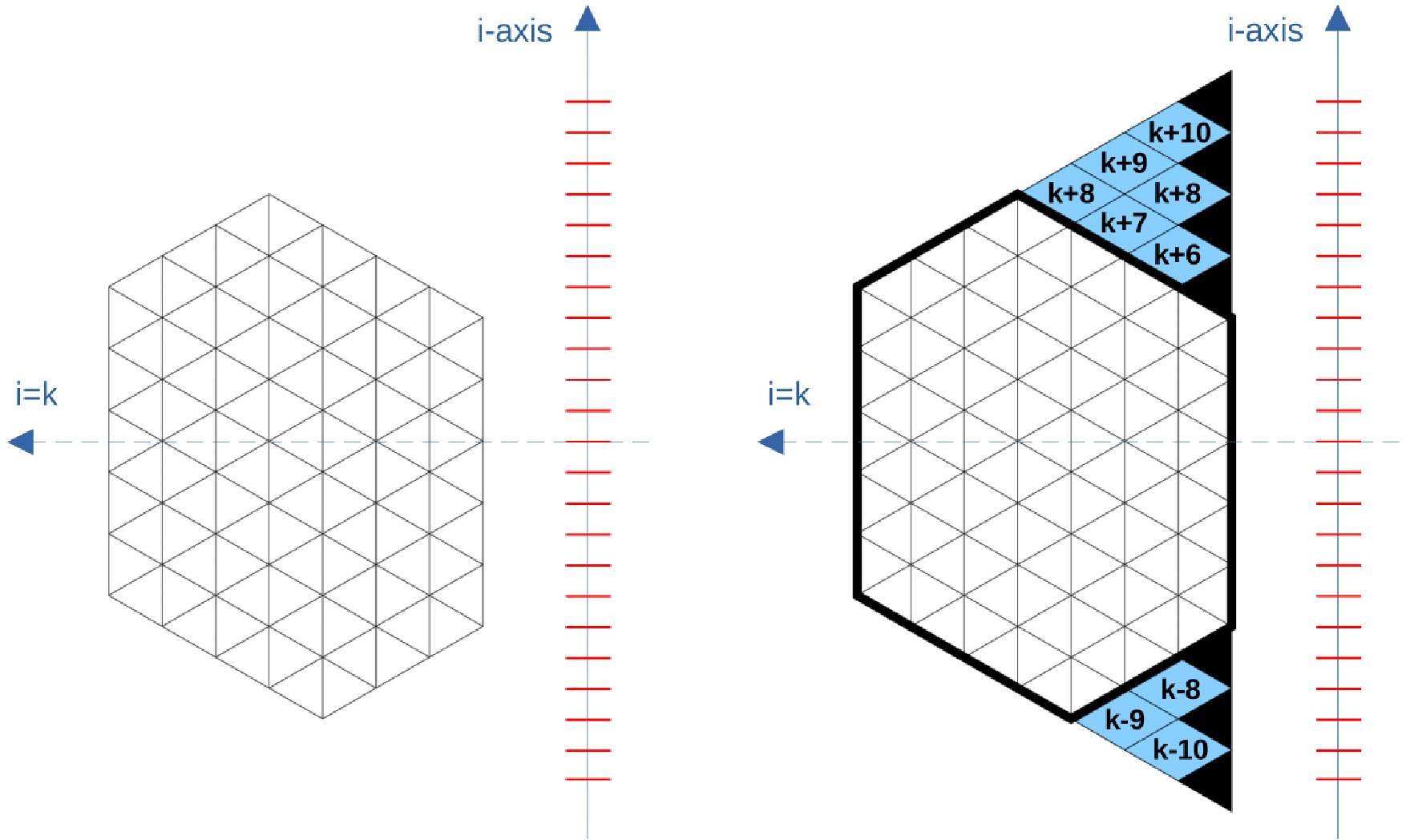}
        \caption{$H(5,3,4)$ (left) and $S_{5,7}(1,2,3,9,10,11,12)$ (right) on $(i,j)$-coordinate planes.}
        \label{ffd}
    \end{figure}

    First, note that the region $H_{a,b,c}$ has the same number of lozenge tilings as the region $S_{a,b+c}\Big(\overline{[b]\cup\big([a+b+c]\setminus[a+b]\big)}\Big)$. This is because, after removing every forced lozenge from the corners of the latter region, one obtains the hexagonal region $H_{a,b,c}$ (compare the two pictures in Figure \ref{ffd}). However, since all the forced lozenges are horizontal, they have weights different from $1$ and thus the tiling generating functions of the two regions differ by the product of weights of all forced lozenges. After keeping track of their weights, we get
    \begin{equation}\label{efe}
    \begin{aligned}
        &\frac{\M_{(q,X,Y,k)}\Big(S_{a,b+c}\Big(\overline{[b]\cup\big([a+b+c]\setminus[a+b]\big)}\Big)\Big)}{\M_{(q,X,Y,k)}(H_{a,b,c})}\\
        =&\prod_{b_{1}<b_{2}}\frac{Xq^{b_{1}+b_{2}+k}+Yq^{-(b_{1}+b_{2}+k)}}{2}\cdot\prod_{c_{1}<c_{2}}\frac{Xq^{c_{1}+c_{2}+k}+Yq^{-(c_{1}+c_{2}+k)}}{2},
    \end{aligned}
    \end{equation}
    where the first product on the right side of \eqref{efe} is taken over $b_{1},b_{2}\in[b]-\frac{a+b+c+1}{2}$ and the second product is taken over $c_{1},c_{2}\in[a+b+c]\setminus[a+b]-\frac{a+b+c+1}{2}$. Combining \eqref{efa}, \eqref{efe}, and applying algebraic manipulation leads us to \eqref{efd}. This completes the proof.
\end{proof}
\end{document}